\documentclass[11pt]{amsart}
\usepackage{mathabx}
\usepackage{bbm,url}
\usepackage{graphicx}
\usepackage{fullpage}
\usepackage{amsmath,amsthm,amssymb,latexsym,color, dsfont, hyperref, calc, accents, mathtools}

\usepackage[usenames,dvipsnames,svgnames,table]{xcolor}
\usepackage{bbm}
\usepackage{bm}
\usepackage{float}
\usepackage{caption}
\usepackage{subcaption}

\def\R{{\mathbb R}}

\def\E{{\mathbb E}}
\renewcommand{\P}{\mathbb{P}}

\def\X{{\mathcal X}}
\def\S{{\mathcal S}}

\def\V{{\mathcal V}}
\def\Net{{\mathcal N}}
\def\Y{{\mathcal Y}}

\newcommand{\vct}[1]{\bm{#1}}

\newcommand{\eps}{\varepsilon}

\renewcommand{\epsilon}{\varepsilon}
\newcommand{\abs}[1]{\left|#1\right|}
\newtheorem{theorem}{Theorem}
\newtheorem{definition}{Definition}
\newtheorem{lemma}{Lemma}
\newtheorem{proposition}{Proposition}
\newtheorem{remark}{Remark}

\newtheorem{corollary}{Corollary}

\DeclareMathOperator*{\argmin}{arg\,min}
\renewcommand{\span}{\text{span}}
\newcommand{\vect}{\text{vect}}
\newcommand{\polylog}{\text{polylog}}
\DeclareMathOperator*{\outprod}{\bigcirc}

\definecolor{mypink3}{cmyk}{0, 0.7808, 0.4429, 0.1412}

\title{Modewise Operators, the Tensor Restricted Isometry Property, and Low-Rank Tensor Recovery}\thanks{
\noindent Mark A. Iwen partially supported by NSF DMS 1912706.  Deanna Needell partially supported by NSF DMS 2011140. Mark A. Iwen, Deanna Needell, and Elizaveta Rebrova partially supported by NSF DMS 2106472.} 

\author{Mark A. Iwen}
\address{Michigan State University, Department of Mathematics, and the Department of Computational Mathematics, Science and Engineering (CMSE).}
\email{markiwen@math.msu.edu} 

\author{Michael Perlmutter}
\address{University of California, Los Angeles, Department of Mathematics}
\email{perlmutter@math.ucla.edu}

\author{Deanna Needell}
\address{University of California, Los Angeles, Department of Mathematics}
\email{deanna@math.ucla.edu}

\author{Elizaveta Rebrova}
\address{Princeton University, Department of Operations Research and Financial Engineering}
\email{elre@princeton.edu}

\begin{document}

\maketitle

\begin{abstract}
Recovery of sparse vectors and low-rank matrices from a small number of linear measurements is well-known
to be possible under various model assumptions on the measurements. The key requirement on the measurement matrices is typically the restricted isometry property, that is, approximate orthonormality when acting on the subspace to be recovered. Among the most widely used random matrix measurement models are (a) independent sub-gaussian models and (b) randomized Fourier-based models, allowing for the efficient computation of the measurements.

For the now ubiquitous tensor data, direct application of the known recovery algorithms to the vectorized or matricized tensor is awkward and memory-heavy because of the huge measurement matrices to be constructed and stored. In this paper, we propose modewise measurement schemes based on sub-gaussian and randomized Fourier measurements. These modewise operators act on the pairs or other small subsets of the tensor modes separately. They require significantly less memory than the measurements working on the vectorized tensor, provably satisfy the tensor restricted isometry property and experimentally can recover the tensor data from fewer measurements and do not require impractical storage.
\end{abstract}

\section{Introduction and prior work}\label{sec: intro}

Geometry preserving
dimension reduction has become important in a wide variety of applications in the last two decades due to improved sensing capabilities and the increasing prevalence of massive data sets.   This is motivated in part by the fact that the data one collects often consists of high-dimensional representations of intrinsically simpler and effectively lower-dimensional data.  In such settings, randomized linear projections have been demonstrated to preserve the intrinsic geometric structure of the collected data in a wide range of applications in both computer science (where one often deals with finite data sets \cite{dasgupta1999elementary,achlioptas2003database}) and signal processing (where manifold \cite{baraniuk_random_2009} and sparsity \cite{foucart_mathematical_2013} assumptions are common).  In this context, the vast majority of prior work has been focused on recovering vector data taking values in a set $S \subset \mathbb{R}^n$ using random linear maps into $\mathbb{R}^m$ with $m \ll n$ which are guaranteed to approximately preserve the norms of all elements in $S$. The focus of this paper is extending this line of work to  higher-order
tensors taking values in $\mathbb{R}^{n_1\times\ldots\times n_d}$.

In the vector case, uniform guarantees for the approximate norm preservation for all sparse vectors, in the form of the restricted isometry property (RIP), have numerous applications. They include recovery algorithms that reconstruct all sparse vectors from a few linear measurements (such as,  $l_1$-minimization \cite{candes2006stable,foucart2012sparse,mo2011new}, orthogonal matching pursuit \cite{zhang2011sparse}, CoSaMP \cite{needell2009cosamp, foucart2012sparse}, iterative hard thresholding \cite{blumensath2009iterative} and hard
thresholding pursuit \cite{foucart2011hard}). Extending these algorithms from sparse vector recovery to low-rank matrix or low-rank tensor recovery is very natural. Indeed, rank-$r$ matrices (i.e., two-mode tensors) in $\mathbb{R}^{n \times n}$  can be recovered from $\mathcal{O}(rn)$ linear measurements \cite{candes2009exact,foucart_mathematical_2013}.  Extensions to the low-rank higher-order tensor setting,
however, are less straightforward  due to, e.g., the more complicated structure of higher-order singular value decomposition and non-unique definition of the tensor rank. Still, there are many applications that motivate the
use of tensors, ranging from video and longitudinal imaging \cite{liu2012tensor, bengua2017efficient} to machine learning \cite{romera2013multilinear, vasilescu2005multilinear} and
differential equations \cite{beck2000multiconfiguration, lubich2008quantum}.  Thus, while tensor applications are ubiquitous and moreover the tensors arising in these applications are extremely large-scale, few methods exist that do satisfactory tensor dimension reduction.  Our goal here is thus to demonstrate a tensor dimension reduction technique that is computationally feasible (in terms of application and storage) and that guarantees preservation of geometry.  As a motivating example, we consider the problem of tensor reconstruction from such dimension reduction measurements, and in particular the Tensor Iterative Hard Thresholding method is used for this purpose herein.

In \cite{rauhut2017low}, the authors propose tensor extensions of the Iterative Hard Thresholding (IHT) method for several tensor decomposition formats, namely the higher-order singular value decomposition (HOSVD), the tensor train format, and the general hierarchical Tucker decomposition. Additionally, the recent papers \cite{SWIM19siht,grotheer2019iterative} extend the Tensor IHT method (TIHT) to low Canonical Polyadic (CP) rank and low Tucker rank tensors, respectively.  TIHT as the name suggests is an iterative method that consists of one step that applies the adjoint of the measurement operator to the remaining residual and a second step that thresholds that signal proxy to a low-rank tensor.  This method has seen provable guarantees for reconstruction under various geometry preserving assumptions on the measurement maps \cite{rauhut2017low,SWIM19siht,grotheer2019iterative}.   All these works however propose first reshaping a $d$-mode tensor $\mathcal{X}\in\mathbb{R}^{n_1\times\ldots\times n_d}$ into an $\prod_{i=1}^dn_i$-dimensional vector $\mathbf{x}$ and then multiplying by an $m\times \prod_{i=1}^dn_i$ matrix $A$. Unfortunately, this means that the matrix $A$ must be even larger than the original tensors $\mathcal{X}$. The main goal of this paper is to propose a more memory-efficient alternative to this approach.

In particular, we propose a \emph{modewise} framework for low-rank tensor recovery.  A general two-stage modewise linear operator $\mathcal{L}: \mathbb{R}^{n_1 \times \dots \times n_d} \rightarrow \mathbb{R}^{m_1 \times \dots \times m_{d'}}$ takes the form 
\begin{equation}
\label{equ:GeneralModewiseOp}
\mathcal{L}(\mathcal{X}) := \mathcal{R}_2 \left(\mathcal{R}_1 (\mathcal{X}) \times_1 A_1 \dots \times_{\tilde d} A_{\tilde d} \right) \times_1 B_1 \dots \times_{d'} B_{d'},
\end{equation}
where $(i)$ $\mathcal{R}_1$ is a reshaping operator which reorganizes an $\mathbb{R}^{n_1 \times \dots \times n_d}$ tensor into an $\mathbb{R}^{\tilde{n}_1 \times \dots \times \tilde{n}_{\tilde d}}$ tensor, after which $(ii)$ each $A_j \in \mathbb{R}^{\tilde{m}_j \times \tilde{n}_j}$ is applied to the resphaped tensor for $j = 1, \dots, \tilde{d}$ via a modewise product (reviewed in Section \ref{sec: prereqs}), followed by $(iii)$ an additional reshaping via $\mathcal{R}_2$ into an $\mathbb{R}^{m'_1 \times \dots \times m'_{d'}}$ tensor, and finally $(iv)$ additional $j$-mode products with the matrices $B_j \in \mathbb{R}^{m_j \times m'_j}$ for $j = 1, \dots, d'$.  More general $n$-stage modewise operators can be defined similarly.  First analyzed in \cite{iwen2019lower,jin2019faster} for aiding in the rapid computation of the CP decomposition, such modewise compression operators offer a wide variety of computational advantages over standard vector-based approaches (in which ${\mathcal R}_1$ is a vectorization operator so that $\tilde d = 1$, $A_1 = A \in {\mathbb{R}^{m \times \prod_j^d n_j}}$ is a standard Johnson-Lindenstrauss map, and all remaining operators $\mathcal{R}_2, B_1, \dots$ are the identity).  In particular, when $\mathcal{R}_1$ is a more modest reshaping (or even the identity) the resulting modewise linear transforms can be formed using significantly fewer random variables (effectively, independent random bits), and stored using less memory by avoiding the use of a single massive  $m \times \prod_j^d n_j$ matrix.  In addition, such modewise linear operators also offer trivially parallelizable operations, faster serial data evaluations than standard vectorized approaches do for structured data (see, e.g., \cite{jin2019faster}), and the ability to better respect the multimodal structure of the given tensor data.  

{\bf Related Work and Contributions:}  This paper is directly motivated by recent work on low-rank tensor recovery using vectorized measurements  \cite{rauhut2017low,SWIM19siht,grotheer2019iterative}.  Given the framework that modewise measurement operators \eqref{equ:GeneralModewiseOp} provide for creating highly structured and computationally efficient measurement maps, we aim to provide both theoretical guarantees and empirical evidence that several modewise maps allow for the efficient recovery of tensors with low-rank HOSVD decompositions. This represents the first study of such modewise maps for performing norm-preserving dimension reduction of nontrivial infinite sets of elements in (tensorized) Euclidean spaces, and so provides a general framework for generalizing the use of such maps to other types of, e.g., low-rank tensor models.  See Section~\ref{sec: statement of main results} for the specifics of our theoretical results as well as Section~\ref{sec: IHT} for an empirical demonstration of the good performance such modewise maps can provide for tensor recovery in practice.

Other recent work involving the analysis of modewise maps for tensor data include, e.g., applications in kernel learning methods which effectively use modewise operators specialized to finite sets of rank-one tensors \cite{ahle2020oblivious}, as well as a variety of works in the computer science literature aimed at compressing finite sets of low-rank (with respect to, e.g., CP and tensor train decompositions \cite{rakhshan2020tensorized}) tensors. More general results involving extensions of bounded orthonormal sampling results to the tensor setting \cite{jin2019faster, bamberger2021johnson} apply to finite sets of arbitrary tensors.  With respect to norm-preserving modewise embeddings of infinite sets, prior work has been limited to oblivious subspace embeddings (see, e.g., \cite{iwen2019lower, malik2020guarantees}).  Here, we extend these techniques to the set of all tensors with a low-rank HOSVD decomposition in order to obtain modewise embeddings with the Tensor Restricted Isometry Property (TRIP).  Having obtained modewise TRIP operators, we then consider low-rank tensor recovery via Tensor IHT (TIHT).

{\bf Paper Outline:}  The rest of this paper is organized as follows.  In Section \ref{sec: prereqs}, we will provide a brief review of basic tensor definitions. In Section \ref{sec: statement of main results}, we will state our main results, which we then prove in Section \ref{sec: main proofs}. In Section \ref{sec: IHT}, we discuss applications of our results recovering low-rank tensors via the TIHT, and in present numerical results. In Section \ref{sec: conclusion}, we provide a short conclusion and discussion of directions for future work. Proofs of auxiliary results are provided in the appendices.

\section{Tensor prerequisites} \label{sec: prereqs}

In this section, we briefly review some basic definitions concerning tensors. For further overview, we refer the reader to  \cite{kolda2009tensor}. Let $d\geq 1$, $n_1,\ldots,n_d\geq 1$ be integers, and $[n_j] \coloneqq \{ 1, \dots, n_j \}$ for all $j = 1, \dots, d$.  For a multi-index ${\bf i} = (i_1, \dots, i_d) \in [n_1] \times \dots \times [n_d]$, we will denote the $\mathbf{i}$-th entry of a $d$-mode tensor  $
\X \in \mathbb{R}^{n_1\times\ldots \times n_d}$ by $\mathcal{X}_{\bf i}$.  When convenient we wil also denote the entries by $  \mathcal{X}(i_1, \dots, i_d)$,  $\mathcal{X}_{i_1, \dots, i_d}$, or  $\mathcal{X}({\bf i})$.  For the remainder of this work, we will use bold text to denote vectors (i.e., one-mode tensors), capital letters to denote matrices (i.e., two-mode tensors) and use calligraphic text for all other tensors. 

\subsection{Modewise multiplication and $j$-mode products:}  For $1\leq j\leq d$, the \textit{$j$-mode product} of $d$-mode tensor $\mathcal{X} \in \mathbb{R}^{n_1 \times \dots \times n_{j-1} \times n_j \times n_{j+1} \times \dots \times n_d}$ with a matrix $U \in \mathbb{R}^{m_j \times n_j}$ is another $d$-mode tensor $\mathcal{X} \times_j U \in \mathbb{R}^{n_1 \times \dots \times n_{j-1} \times m_j \times n_{j+1} \times \dots \times n_d}$.  Its entries are given by 
\begin{equation}
(\mathcal{X} \times_j U)_{i_1, \dots, i_{j-1}, \ell, i_{j+1}, \dots, i_d} = \sum_{i_{j}=1}^{n_j} \mathcal{X}_{i_{1},\dots,i_j,\dots,i_{d}} U_{\ell,i_{j}}
\label{equ:DefModejProduct}
\end{equation}
for all $(i_1, \dots, i_{j-1}, \ell, i_{j+1}, \dots, i_d) \in [n_1] \times \dots \times [n_{j-1}] \times [m_j] \times [n_{j+1}] \times \dots \times [n_d]$.
If $\mathcal{Y}$ is a tensor of the form $\mathcal{Y} = \outprod_{i=1}^d {\bf y}_i \in \mathbb{R}^{n_1 \times \dots \times n_d}$,  where $\outprod$ denotes the outer product, then one may use \eqref{equ:DefModejProduct} to see that 
\begin{equation}\label{eqn:modewise action}\mathcal{Y} \times_j U = \left( \outprod_{i=1}^d {\bf y}_i \right) \times_j U = \left( \outprod_{i=1}^{j-1} {\bf y}_i \right) \outprod U {\bf y}_j \outprod \left( \outprod_{i=j+1}^{d} {\bf y}_i \right).\end{equation}
For further discussion of the properties of modewise products, please see \cite{iwen2019lower, kolda2009tensor}.

\subsection{HOSVD decomposition and multilinear rank}
Let $\mathbf{r}=(r_1,\ldots, r_d)$ be a multi-index in $[n_1] \times \dots \times [n_d]$. 
We say that a $d$-mode tensor $\mathcal{X}$ has \textbf{multilinear rank} or \textbf{HOSVD rank} at most $\mathbf{r}$ if there exist subspaces $\mathcal{U}_1 \subset \mathbb{R}^{n_1},\ldots, \mathcal{U}_d\subset \mathbb{R}^{n_d}$ such that 
 \begin{equation*}
 \text{dim}\: \mathcal{U}_i = r_i
 \quad \text{ and } \quad
\mathcal{X}\in \bigotimes_{i=1}^d \mathcal{U}_i,
\end{equation*}
where $\bigotimes_{i=1}^d \mathcal{U}_i$ denotes the tensor product of the subspaces $\mathcal{U}_i$ here.
 We note that a tensor $\mathcal{X}$ has rank at most $\mathbf{r}$ if and only if there exists a core tensor $\mathcal{C} \in \mathbb{R}^{r_1 \times \dots \times r_d}$ such that  
\begin{equation}\label{eqn: HOSVD basis}
    \mathcal{X} =\mathcal{C}\times_1 U^1\times_2\ldots \times_d U^d
    =  
    \sum_{k_d=1}^{r_d}\ldots\sum_{k_1=1}^{r_1}
    {\mathcal C}(k_1,\ldots k_d) \outprod_{i=1}^d {\bf u}_{k_i}^i,
\end{equation}
where, for each $1\leq i \leq d$, $\mathbf{u}^i_1,\ldots,\mathbf{u}^i_{r_i}$ is an orthonormal basis for $\mathcal{U}_i,$ and  $U^i$ is the $n_i\times r_i$ matrix $U^i=(\mathbf{u}^i_1,\ldots,\mathbf{u}^i_{r_i})$.
A factorization of the form \eqref{eqn: HOSVD basis} is called a \textbf{Higher-Order Singular Value Decomposition (HOSVD)} of the tensor  $\mathcal{X}.$ 
It is well-known (see e.g., \cite{rauhut2017low}) that we may assume that the core tensor $\mathcal{C}$ has  orthogonal subtensors in the sense that for all $1 \leq i \leq d$, we have    $\langle {\mathcal C}_{k_i=p}, {\mathcal C}_{k_i=q}\rangle=0$ for all $p\neq q$, i.e. 
\begin{equation}\label{eqn: orthogonal subtensors}
\sum_{\substack{k_j=1 \\ j\neq i}}^{r_j} {\mathcal C}(k_1,\ldots,p, \ldots, k_d) {\mathcal C}(k_1,\ldots,q, \ldots, k_d) = 0 \quad \text{ unless } p=q.
\end{equation}
 We also note that, since each of the $\{{\bf u}^i_{k_i}\}_{i=1}^{r_i}$ form an orthonormal basis for $\mathcal{U}_i$, we have $\|\mathcal{C}\|_F=\|\mathcal{X}\|_F \coloneqq \sqrt{\langle \mathcal{X}, \mathcal{X} \rangle}$, where here $\langle\cdot,\cdot\rangle$ denotes the trace inner product.
\begin{remark} The {\bf Canonical Polyadic (CP) rank} of a $d$-mode tensor is the minimum number of rank-one tensors (i.e., outer products of $d$ vectors) required to represent the tensor as a sum.  If $\mathcal{X}$ has HOSVD rank $\mathbf{r}$ then \eqref{eqn: HOSVD basis} implies $\mathcal{X}$ has CP rank at most $\Pi_{i=1}^d r_i$.(In particular, if $r_i=r$ for all $i$  , then $\mathcal{X}$ has CP rank at most $r^d$.)
\end{remark}

\subsection{Restricted Isometry Properties and Tensors}

\begin{definition}\label{def: trip}[TRIP$(\delta,\mathbf{r})$ property] 
We say that a linear map $\mathcal{A}$ has the TRIP$(\delta,\mathbf{r})$ property if for all $\mathcal{X}$ with HOSVD rank at most $\mathbf{r}$ we have
    \begin{equation}\label{eqn: delta-r-TRIP}
        (1-\delta)\|\mathcal{X}\|^2 \leq \|\mathcal{A}(\mathcal{X})\|^2 \leq(1+\delta)\|\mathcal{X}\|^2    
        \end{equation}
\end{definition}

\begin{definition}\label{def: trip_on_set}[RIP$(\eps, \mathcal{S})$ property] 
We say that a linear map $\mathcal{A}$ has the RIP$(\eps, \mathcal{S})$ property if for all elements $s \in \mathcal{S}$
    \begin{equation}\label{eqn: s-RIP}
        (1-\eps)\|s\|^2 \leq \|\mathcal{A}(s)\|^2 \leq(1+\eps)\|s\|^2    
        \end{equation}
\end{definition}
We emphasize that the  set $\mathcal{S}$ in Definition \ref{def: trip_on_set} can be a subset of any vector space (not necessarily a tensor space). 

\subsection{Reshaping and the HOSVD}\label{3.1}
For simplicity we will assume below, and for the rest of this paper, that there exist $n,r\geq 1$ such that  $n_i=n, r_i=r$ for $1\leq i \leq d$. We note that this assumption is made only for the sake of clarity, and all of our analysis can be extended to the general case. 

We let $\kappa$ be an integer which divides $d$ and let $d'\coloneqq d/\kappa.$ Consider the \emph{reshaping operator}  $$\mathcal{R}:\bigotimes_{i=1}^d\mathbb{R}^{n}\rightarrow \bigotimes_{i=1}^{d'}\mathbb{R}^{n^{\kappa}}$$ that flattens every $\kappa$ modes of a tensor into one. Note that $\mathcal{R}$ decreases the total number of modes from $d$ to $d' = d/\kappa$. Formally, $\mathcal{R}$ is defined to be the unique linear operator such that on rank-one tensors it acts as 
\begin{equation*}
    \mathcal{R} \left(\outprod_{i=1}^d\mathbf{x}^i \right) \coloneqq \outprod_{i=1}^{d'}\left(\bigotimes_{\ell=1+\kappa(i-1)}^{\kappa i}\mathbf{x}^\ell\right)\eqqcolon
     \outprod_{i=1}^{d'}\accentset{\circ}{\mathbf{x}}^{i},
\end{equation*}
where $\otimes$ denotes the Kronecker product when applied to vectors.
We observe that if a tensor $\mathcal{X}$ has a form~\eqref{eqn: HOSVD basis}, then its reshaping $\accentset{\circ}{\mathcal{X}} \coloneqq \mathcal{R}(\mathcal{X})$  is the $d'$-mode tensor $\accentset{\circ}{\mathcal{X}}\in\bigotimes_{i=1}^{d'}\R^{n^{\kappa}} $ with HOSVD rank at most  $\mathbf{r'} \coloneqq(r^{\kappa},\ldots, r^{\kappa})$ given by \begin{equation}\label{eqn: reshape HOSVD}
    \accentset{\circ}{\mathcal{X}} = \sum_{j_{d'}=1}^{r^{\kappa}}\ldots\sum_{j_1=1}^{r^{\kappa}} \accentset{\circ}{\mathcal C}(j_1,\ldots,j_{d'}) \outprod_{\ell=1}^{d'}\accentset{\circ}{\mathbf{u}}^{\ell}_{j_{\ell}},
    \end{equation}
where the new component vectors $\accentset{\circ}{\mathbf{u}}^{\ell}_{j_{\ell}}$ are obtained by taking Kronecker product of the appropriate $\mathbf{u}^i_{k_i}$, and where $\accentset{\circ}{\mathcal C} \in \mathbb{R}^{r^\kappa \times \dots \times r^\kappa}$ is a reshaped version of $\mathcal{C}$ from \eqref{eqn: HOSVD basis}.
Since each of the $\{\mathbf{u}^i_{k_i}\}_{k_i=1}^{r}$ was an orthonormal basis for $\mathcal{U}_i,$ it follows that $\{\accentset{\circ}{\mathbf{u}}^i_{j_i}\}_{j_i=1}^{r^{\kappa}}$ is an 
orthonormal basis for $\accentset{\circ}{\mathcal{U}}_i \coloneqq \span \left(\{\accentset{\circ}{\mathbf{u}}^i_{j_i}\}_{j_i=1}^{r^\kappa} \right)$. 

\section{Main results: modewise TRIP}\label{sec: statement of main results}

For $1\leq i \leq d'$, let $A_i$ be an $m\times n^\kappa$ matrix, let 
$\mathcal{A}: \R^{n^{\kappa}\times\ldots\times n^{\kappa}}\rightarrow \R^{m\times\ldots\times m}$ be the linear map     which acts modewise on $d'$-mode tensors by 
\begin{equation}\label{eqn: A}
\mathcal{A}(\mathcal{Y})= \mathcal{Y}\times_1 A_1\times_2\ldots\times_{d'} A_{d'}.
\end{equation}
Let $\mathcal{X}$ be a $d$ mode tensor with  HOSVD decomposition given by \eqref{eqn: HOSVD basis}. By  \eqref{eqn:modewise action} and \eqref{eqn: reshape HOSVD}, we have that 
\begin{equation}\label{eqn: newfactors}
    \mathcal{A}(\mathcal{R}(\mathcal{X})) =\mathcal{A}(\accentset{\circ}{\mathcal{X}}) = \sum_{j_{d'}=1}^{r^{\kappa}}\ldots\sum_{j_1=1}^{r^{\kappa}} \accentset{\circ}{C}(j_1,\ldots,j_{d'}) (A_1\accentset{\circ}{\mathbf{u}}^1_{j_1} \circ\ldots \circ A_{d'}\accentset{\circ}{\mathbf{u}}^{d'}_{j_{d'}}).
\end{equation}
Our first main result will show that $\mathcal{A}$ satisfies the TRIP$(\delta,\mathbf{r})$ property under the assumption that each of the $A_i$ satisfies a restricted isometry property on the set $\S_{1,2}$ defined below.

\begin{definition}\label{def: S}[The set $\S_{1,2}$] Consider a set of vectors in $\mathbb{R}^{n^\kappa},$
\begin{equation}\label{s1}
S_1\coloneqq\{\accentset{\circ}{\bf u}~|~ \accentset{\circ}{\bf u} = \otimes_1^\kappa {\bf u}^i, {\bf u}^i\in\mathbb{S}^{n-1}\},
\end{equation}
and let
 $\S_2 \coloneqq \left\{ \frac{{\bf x} + {\bf y}}{\| {\bf x} + {\bf y} \|_2} ~\big|~ {\bf x}, {\bf y} \in \S_1 ~{\rm s.t.}~ \langle {\bf x}, {\bf y} \rangle = 0 \right\}.$
For the rest of this text we will let  $\S_{1,2}\coloneqq\S_1\cup \S_2$, and note that $\S_{1,2} \subseteq \mathbb{S}^{n^\kappa-1}$.
\end{definition}
More precisely, will show that $\mathcal{A}\circ\mathcal{R}$ satisfies the TRIP$(\delta,\mathbf{r})$ property under the assumption that each of the $A_i$ satisfy the RIP($\eps$, $\S_{1,2}$), where $\eps$ is a suitably chosen parameter depending on $\delta$. In the case where $\mathbf{r}=\mathbf{1}\coloneqq(1,1,\ldots,1)$, this is nearly trivial. Indeed, if $\accentset{\circ}{\mathcal{X}} = c \outprod_{\ell=1}^{d'}\accentset{\circ}{\mathbf{u}}^{\ell}
    $, and $\mathcal{A}$ is the map defined in  \eqref{eqn: A}, then we have 
    \begin{equation*}
\left\|\mathcal{A}( \accentset{\circ}{\mathcal{X}}) \right\| =\left\| c \outprod_{\ell=1}^{d'}A_i\accentset{\circ}{\mathbf{u}}^{\ell}\right\|=|c|\prod_{\ell=1}^d \|A_i\accentset{\circ}{\mathbf{u}}^{\ell}\|.
    \end{equation*}
    Therefore, since $\left\|\accentset{\circ}{\mathcal{X}} \right\| =|c|\prod_{\ell=1}^d \|\accentset{\circ}{\mathbf{u}}^{\ell}\|$,  we immediately obtain  the following proposition.
    \begin{proposition}\label{prop: rank1}
    Suppose that $\mathcal{A}$ is defined as per \eqref{eqn: A} and that each of the $A_i$ have RIP$\left(\eps, \S_{1,2}\right)$ property. Let $\delta = \max\{(1+\eps)^d-1,(1-\eps)^d+1\}$ and assume that $\delta < 1$.
Then $\mathcal{A}\circ\mathcal{R}$ satisfies the TRIP$(\delta,\mathbf{1})$ property, that is,
\begin{equation*}
    (1-\delta)\|\mathcal{X}\|^2
    \leq\|\mathcal{A}(\mathcal{R}(\mathcal{X}))\|^2    \leq (1+\delta)\|\mathcal{X}\|^2
\end{equation*}
for all $\mathcal{X}$ with HOSVD rank $\mathbf{1}=(1,1,\ldots,1)$. \end{proposition}

Our first main result is the following theorem which is the analogue for Proposition  \ref{prop: rank1} for $r\geq 2.$ It shows that if the each of the $A_i$ satisfies RIP$\left(\eps, \S_{1,2}\right)$ property for a suitable value of $\eps$, then $\mathcal{A}$ has the TRIP($\delta,\mathbf{r})$ property.
\begin{theorem}\label{thm: first compression}
Suppose that $\mathcal{A}$ is defined as per \eqref{eqn: A} and that each of the $A_i$ have RIP$\left(\eps, \S_{1,2}\right)$ property. Let $r \ge 2$, let $\delta = 4d'r^d \eps$ and assume that $\delta < 1$.
Then $\mathcal{A}\circ\mathcal{R}$ satisfies the TRIP$(\delta,\mathbf{r})$ property, i.e.,
\begin{equation}\label{delta-trip}
    (1-\delta)\|\mathcal{X}\|^2
    \leq\|\mathcal{A}(\mathcal{R}(\mathcal{X}))\|^2    \leq (1+\delta)\|\mathcal{X}\|^2
\end{equation}
for all $\mathcal{X}$ with HOSVD rank less than $\mathbf{r}=(r,r,\ldots,r)$. 
\end{theorem}

\begin{proof}
See Section~\ref{sec:ProofofTHm1}.
\end{proof}

The following corollary shows that we may pick the matrices $A_i$ to have i.i.d sub-gaussian entries.

\begin{corollary}\label{cor:sub-gaussian model}
Let $r\geq 2$, and let $\mathbf{r}=(r,r,\ldots,r)$. Suppose that $\mathcal{A}$ is defined as per \eqref{eqn: A} and that each of the $A_i \in \R^{m \times n^\kappa}$ has i.i.d. sub-gaussian entries, for all $i = 1, \ldots, d'$, where $d'=d/\kappa$ for $\kappa \geq 2$, and suppose that $0<\eta, \delta<1.$ Let
\begin{equation}\label{m_first_bound}
m \ge C \delta^{-2}r^{2d} \max\left\{\frac{n d^2\ln(\kappa)}{\kappa}, \frac{d^2}{\kappa^2} \ln\left(\frac{d}{\kappa \eta}\right)\right\}
\end{equation}
for a sufficiently large constant $C$. Then $\mathcal{A}\circ\mathcal{R}$ satisfies TRIP($\delta,\mathbf{r}$) property~\eqref{eqn: delta-r-TRIP} with probability at least $1 - \eta$.
\end{corollary}

\begin{proof}
See Section~\ref{sec:ProofofTHm1}.
\end{proof}

For another possible choice of the $A_i,$ we consider the set of Subsampled Orthogonal  with Random Sign matrices defined below. Note, in particular, that this class includes subsampled Fourier (i.e., discrete cosine and sine) matrices.

\begin{definition}[Subsampled Orthogonal with Random Sign (SORS) matrices]\label{SORSdef} Let $F\in\R^{n\times n}$ 
denote an orthonormal matrix obeying
\begin{equation}\label{BOS}
F^*F=I\quad\text{and}\quad\max_{i,j}\abs{F_{ij}}\le \frac{\Delta}{\sqrt{n}}
\end{equation}
for some $\Delta>0.$ 
Let $H \in\R^{m\times n}$ be a matrix whose rows are chosen i.i.d. uniformly at random from the rows of $F$. We define a Subsampled Orthogonal with Random Sign (SORS) measurement ensemble as $A= \sqrt{\frac{n}{m}} H D$, where $D\in\R^{n\times n}$ is a random diagonal matrix whose the diagonal entries are i.i.d.~$\pm 1$ with equal probability.
\end{definition}

Analogous to Corollary  \ref{cor:sub-gaussian model}, the following result shows that we may choose our matrices $A_i$ to be SORS matrices in Theorem \ref{thm: first compression}.

\begin{corollary}\label{cor:sors}
Let $r\geq 2$ and let $\mathbf{r}=(r,r,\ldots,r)$. Suppose that $\mathcal{A}$ is defined as per \eqref{eqn: A} and that each of the $A_i \in \R^{m \times n^\kappa}$  is a SORS matrix with $\Delta \leq C'$ for a universal constant $C'$, as per Definition~\ref{SORSdef}, for all $i = 1, \ldots, d'$, where $d'=d/\kappa$ for $\kappa \geq 2$.  Furthermore, suppose that $0<\eta, \delta<1.$ Let
\begin{equation}\label{m_first_bound_sors}
m\geq C_1\delta^{-2}r^{2d}\frac{n d^2\ln(\kappa)}{\kappa} \cdot L,
\end{equation}
where
\begin{equation}\label{log-factors}
L = \ln\left(\frac{2d}{\kappa \eta}\right) \ln\left(\frac{2en^{\kappa}d}{\kappa \eta}\right) \ln^2\left[C_2 \delta^{-2} r^{2d} \frac{ n d^2 \ln(\kappa)}{\kappa}\ln\left(\frac{2 d}{\kappa \eta}\right)\right]
\end{equation}
and $C_1, C_2$ are sufficiently large absolute constants. Then $\mathcal{A}\circ\mathcal{R}$ satisfies TRIP($\delta,\mathbf{r}$) property~\eqref{eqn: delta-r-TRIP} with probability at least $1 - \eta$.
\end{corollary}

\begin{proof}
See Section~\ref{sec:ProofofTHm1}.
\end{proof}

To further improve embedding dimension of $m^{d'}$ provided by Corollaries  $\ref{cor:sub-gaussian model}$ and~\ref{cor:sors}, we can apply a secondary compression, analogous to the one used in \cite{iwen2019lower}, by letting
\begin{equation}\label{a final}
    \mathcal{A}_{2nd}(\mathcal{X})\coloneqq A_{\text{2nd}}(\vect(\mathcal{A}(\mathcal{R}(\mathcal{X}))),
\end{equation}
where $\vect$ is a vectorization operator.  In this case, we again wish to show that $\mathcal{A}_{2nd}$ satisfies $TRIP(\delta,\mathbf{r})$ for suitably chose parameters.  One of the key challenges in doing this is that,
 for any given $i$, the new factor vectors $\{A_{i}\accentset{\circ}{u}^{i}_{j_{i}}\}_{j_i=1}^{r^\kappa}$ defined as in \eqref{eqn: newfactors} are no longer orthogonal to one another.  Therefore \eqref{eqn: newfactors} is not an HOSVD decomposition of  $\mathcal{A}(\accentset{\circ}{\mathcal{X}}$), and the HOSVD rank of $\mathcal{A}(\accentset{\circ}{\mathcal{X}})$ might be much larger than the HOSVD rank of $\accentset{\circ}{\mathcal{X}}$.
However, one may overcome this difficulty by observing that, with high probability, $\mathcal{A}(\mathcal{R}(\mathcal{X}))$ will belong to the following set of nearly orthogonal tensors.

\begin{definition}[Nearly orthogonal tensors $\mathcal{B}_{R,\mu,\theta,\mathbf{r}}$]\label{def: fancy B}
Let $\mathcal{B}_{R,\mu,\theta,\mathbf{r}}$ be the set of d-mode tensors in $\mathcal{X} \in \R^{n \times \ldots \times n}$  that may be written in  standard form \eqref{eqn: HOSVD basis} 
such that 
\begin{enumerate}
    \item[(a)] $\|{\bf u}^i_{k_i}\|_2\leq R$  for all $i$ and $k_i$, 
    \item[(b)]  $|\langle {\bf u}^i_{k_i}, {\bf u}^i_{k'_i}\rangle| \leq \mu$ for all $k_i\neq k'_i$,
    \item[(c)] the core tensor $\mathcal{C} $ satisfies $\|\mathcal{C}\|_F = 1$,
    \item[(d)] $\mathcal{C}$ has orthogonal subtensors in 
    the sense that \eqref{eqn: orthogonal subtensors} holds for all $1\leq i\leq d$,
    \item[(e)] $\|\mathcal{X}\|_F\geq \theta$.
\end{enumerate}
 \end{definition}

Our next main result is the following theorem which shows that $\mathcal{A}_{2nd}$ satisfies $TRIP(\delta,\mathbf{r})$ for suitably chosen parameters.
 
\begin{theorem}\label{thm: second compression}
Let $r\geq 2$ and let $\mathbf{r}=(r,r,\ldots,r) \in \R^d$. Suppose that $\mathcal{A}$ and  $\mathcal{A}_{2nd}$ are defined as in \eqref{eqn: A} and \eqref{a final}. Let $d'=d/\kappa$ and assume that $A_i$ satisfies the  RIP$\left(\eps, \mathcal{S}_{1,2} \right)$ property  
for all $i=1,\ldots,d',$  where
$\delta = 12d'r^d \eps<1$. 
Assume that $A_{\text{2nd}}$ satisfies the  RIP$\left(\delta/3,\mathcal{B}_{1+\epsilon,\epsilon,1-\delta/3,\mathbf{r'}} \right)$, property where $\mathbf{r'}=(r^\kappa,\ldots,r^\kappa) \in \R^{d'}$.
Then,  $\mathcal{A}_{2nd}$ will satisfy the TRIP($\delta,\mathbf{r}$) property, i.e.,
\begin{equation*}
    (1-\delta)\|\mathcal{X}\|^2\leq \|\mathcal{A}_{2nd}(\mathcal{X})\|^2\leq     (1+\delta)\|\mathcal{X}\|^2
\end{equation*}
for all $\mathcal{X}$ with HOSVD rank at most $\mathbf{r}$.
\end{theorem}

\begin{proof}
See Section~\ref{sec: proof of second compression}.
\end{proof}

The following two corollaries show that we may choose the matrices $A_i$ and $A_\text{2nd}$ to be either sub-gaussian or SORS matrices. We also note that it is possible to produce other variants of these corollaries where, for example, one takes each $A_i$ to be sub-gaussian and lets $A_\text{2nd}$ be a SORS matrix. 

\begin{corollary}
\label{thm: second compression Gaussian}
Let $r\geq 2$ and let $\mathbf{r}=(r,r,\ldots,r)  \in \R^d$. Suppose that $\mathcal{A}$ and  $\mathcal{A}_{2nd}$ are defined as in \eqref{eqn: A} and \eqref{a final}, and that  all of the $A_i \in \R^{m \times n^\kappa}$ have i.i.d. sub-gaussian entries for all $i = 1, \ldots, d'$, where $d'=d/\kappa$, and suppose that $0<\eta, \delta<1.$  Let
\begin{equation}\label{m_first_bound_two_step}
m \ge C \delta^{-2}r^{2d} \max\left\{\frac{n d^2\ln(\kappa)}{\kappa}, \frac{d^2}{\kappa^2} \ln\left(\frac{2d}{\kappa \eta}\right)\right\},
\end{equation}
and let $A_{\text{2nd}} \in \R^{m_{\text{2nd}} \times m}$ also be a sub-gaussian matrix with i.i.d. entries with
\begin{equation}\label{m_second_bound_two_step}
m_{\text{2nd}}\geq C\delta^{-2}\max\left\{ \left(\frac{r^d \kappa + dmr^\kappa}{\kappa} \right)\ln\left(\frac{d}{\kappa}+1\right) + {\frac{dmr^\kappa}{\kappa}}\ln \left(1+\delta r^d\right) + {\frac{d^2mr^\kappa\delta}{ \kappa^{2}}},\ln\left(\frac{2}{\eta}\right)\right\}.
\end{equation}
Then, $\mathcal{A}_{2nd}$ satisfies the TRIP($\delta,\mathbf{r}$) property, i.e.,
\begin{equation*}
    (1-\delta)\|\mathcal{X}\|^2\leq \|\mathcal{A}_{2nd}(\mathcal{X})\|^2\leq     (1+\delta)\|\mathcal{X}\|^2, 
\end{equation*}
for all $\mathcal{X}$ with HOSVD rank at most $\mathbf{r}$ with probability at least $1 - \eta$.
\end{corollary}

\begin{proof}
See Section~\ref{sec: proof of second compression}.
\end{proof}
\begin{remark}
Note that applying the reshaping operator (with $\kappa>1$) is necessary in order for us to actually achieve dimension reduction in the first step. Indeed, if  $\kappa=1,$ then \eqref{m_first_bound_two_step} requires $m>n^\kappa$. We also note that when other parameters are held fixed, the final dimension $m_{2nd}$ will be required to be $\mathcal{O}(n)$, $\mathcal{O}(\delta^{-2})$, $\mathcal{O}(\ln(\eta^{-1}))$ or $\mathcal{O}(r^{2d})$. While the dependence  on the number modes $d$ is exponential, we are primarily interested in cases where $n$ is large in comparison to the rank or the number of modes. In this case, the terms involving $n$ will dominate the terms involving $r^d$.  In Section~\ref{sec: experiments}, we will show that TRIP-dependent tensor recovery methods (e.g., tensor iterative hard thresholding, discussed in Section~\ref{sec: IHT}), successfully work for $d =4$ and $\kappa = 2$. 

In \cite{rauhut2017low}, the author considered i.i.d.\@ sub-gaussian measurements applied to the vectorizations of low-rank tensors and proved that the $TRIP(\delta, \textbf{r})$ property will hold with probability at least $1 - \eta$ if the target dimension satisfies
$$
m_{final} \ge C\delta^{-2} \max\{(r^{d} + dnr) \ln d, \ln(\eta^{-1})\}.
$$
We note this bound has the same computational complexity as ours with respect to $n$, $\delta$, and $\eta$. While their result has much better dependence on $r$, here, we are primarily interested in high-dimensional, low-rank tensors and therefore are primarily concerned with the dependence on $n$.  
\end{remark}

\begin{corollary}

\label{thm: second compression Fourier}
Let $r\geq 2$ and let $\mathbf{r}=(r,r,\ldots,r)$. Suppose that $\mathcal{A}$ and  $\mathcal{A}_{2nd}$ are defined as in \eqref{eqn: A} and \eqref{a final}, and that all of the $A_i \in \R^{m \times n^\kappa}$ are SORS matrices (as per Definition~\ref{SORSdef}) for all $i = 1, \ldots, d'$, where $d'=d/\kappa$. Furthermore, suppose that $0<\eta, \delta<1$ and, as in  \eqref{m_first_bound_sors}, let
\begin{equation}\label{eqn: m_big_first_fouri_ber}
m\geq C_1\delta^{-2}r^{2d}\frac{n d^2\ln(\kappa)}{\kappa} \cdot L, \text{ where } L \text{ is defined by \eqref{log-factors}}.
\end{equation}
Next, let $A_{\text{2nd}} \in \R^{m_{\text{2nd}} \times m}$ also be a SORS matrix with
\begin{equation}\label{eqn: m2_big_Fourier}
m_{\text{2nd}}\geq C\delta^{-2}\left[\frac{r^d \kappa + dmr^\kappa}{\kappa}\ln\left(\frac{d}{\kappa}+1\right) + {\frac{dmr^\kappa}{\kappa}}\ln \left(1+\delta r^d\right) + {\frac{d^2mr^\kappa\delta}{ \kappa^{2}}}\right]\cdot \tilde{L},
\end{equation}
where
$$
\tilde{L} = \ln^2\left(\frac{c_1}{\delta^{2}}\left(\ln\frac{4}{\eta}\right)\left[\frac{r^d \kappa + dmr^\kappa}{\kappa}\ln\left(\frac{d}{\kappa}+1\right) + {\frac{dmr^\kappa}{\kappa}}\ln \left(1+\delta r^d\right) + {\frac{d^2mr^\kappa\delta}{ \kappa^{2}}}\right]\right)\ln\left(\frac{4}{\eta}\right)\ln\left(\frac{4em}{\eta}\right).
$$
Then, $\mathcal{A}_{2nd}$ satisfies the TRIP($\delta,\mathbf{r}$) property, i.e.,
\begin{equation*}
    (1-\delta)\|\mathcal{X}\|^2\leq \|\mathcal{A}_{2nd}(\mathcal{X})\|^2\leq     (1+\delta)\|\mathcal{X}\|^2 
\end{equation*}
holds for all $\mathcal{X}$ with HOSVD rank at most $\mathbf{r}$ with probability at least $1 - \eta$.
\end{corollary}

\begin{proof}
See Section~\ref{sec: proof of second compression}.
\end{proof}

\begin{remark}
Similar to the sub-gaussian case, we note that reshaping (with $\kappa>1$) is needed in order for us to achieve dimension reduction in the first compression. We also note that the final dimension is $\mathcal{O}(n \polylog(n))$, $\mathcal{O}(\delta^{-2}\polylog(\delta^{-2}))$, $\polylog(\eta^{-1})$ and $\mathcal{O}(r^{2d} \polylog(r))$. 
\end{remark}

\section{Low-Rank Tensor Recovery}
\label{sec: IHT}

Low-rank tensor recovery is the task of recovering a low-rank (or approximately low-rank) tensor from a comparatively small number of possibly noisy linear measurements. 
This problem serves as a nice motivating example of where the use of modewise maps with the TRIP$(\delta,\mathbf{r})$ property can help alleviate the burdensome storage requirements of maps which require vectorization.  Indeed, when the goal is compression, storing very large maps in memory as required by vectorization-based approaches is  counterintuitive and often infeasible. 

In the two-mode (matrix) case, the question of low-rank recovery from a small number of linear measurements is now well-known
to be possible under various model assumptions on the measurements  \cite{candes2009exact, candes2010tight, recht2010guaranteed}. One of the standard approaches is so-called nuclear-norm minimization:
\begin{equation*}
\hat{\X} = \argmin_{\X\in\R^{n_1\times n_2}} \|\X\|_* \quad\text{subject to}\quad \mathcal{L}(\X) = y.
\end{equation*}
Since the nuclear norm is defined to be the sum of the singular values, it serves as a fairly good, computationally feasible proxy for rank. 
As in classical compressed sensing, an alternative to optimization-based reconstruction is the use of iterative solvers.  One such approach is the Iterative Hard Thresholding (IHT) method \cite{blumensath2009iterative,blumensath2010normalized,tanner2013normalized} that finds a solution via the alternating updates
\begin{equation}\label{eq:IHT_Y}
\begin{aligned}
&\Y^{j} = \X^j + \mu_j\mathcal{L}^*(y - \mathcal{L}(\X^j)), \\
&\X^{j+1} = \mathcal{H}_{\mathbf{r}}\left(\Y^{j}\right).
\end{aligned}
\end{equation}
where $\X^0$ is initiated randomly. Here, $\mathcal{L}^*$ denotes the adjoint of the operator $\mathcal{L}$, and the function $\mathcal{H}_{\mathbf{r}}$ is a thresholding operator, which returns the closest rank $\mathbf{r}$ matrix.   Results for IHT prove that sparse or low-rank recovery is guaranteed when the measurement operator $\mathcal{L}$ satisfies various properties.  For example, in the case of sparse vector recovery, the restricted isometry property is enough to guarantee accurate reconstruction \cite{blumensath2009iterative}.  In the low-rank matrix case, measurements can be taken to be Gaussian \cite{carpentier2018iterative}, or satisfy various analogues of the restricted isometry property \cite{tanner2013normalized,blanchard2015cgiht,vu2019accelerating}.  In what follows, for the sake of simplicity, we will focus on the case where $\mu_j=1,$  which is referred to as Classical IHT.    However, our results can also be extended to Normalized TIHT where the step size $\mu_j$ takes a different value at each step. (See \cite{rauhut2017low} and the references provided there.) 

The iterative hard thresholding method has been extended to the tensor case (\cite{goulart2015iterative,rauhut2017low, grotheer2019iterative}). In this problem, one aims to recover an unknown tensor $\mathcal{X}\in\mathbb{R}^{n_1\times \ldots\times n_d}$ with e.g., HOSVD rank $\mathbf{r}=(r,\ldots,r)$, where $r \ll \min n_i$, from linear measurements of the form $\mathbf{y}=\mathcal{L}(\mathcal{X})+\mathbf{e}$, where $\mathcal{L}$ is a linear map from $\mathbb{R}^{n_1\times\ldots\times n_d}\rightarrow \mathbb{C}^{m}$, with $m\ll \prod_i n_i,$ and $\mathbf{e}$ is an arbitrary noise vector. The iteration update is given by the same updates as \eqref{eq:IHT_Y}. The primary difference with the matrix case is in the thresholding operator $\mathcal{H}_{\mathbf{r}}$ that approximately computes the best rank $\mathbf{r}$ approximation of a given tensor. Unfortunately, exactly computing the best rank $\mathbf{r}$ approximation of a general tensor is NP-hard. However, it is possible to construct an operator $\mathcal{H}_{\mathbf{r}}$ in  a way such that 
\begin{equation}\label{eq: hr_construction}
    \|\mathcal{Z}-\mathcal{H}_\mathbf{r}(\mathcal{Z})\|_F\leq C\sqrt{d}\|\mathcal{Z}-\mathcal{Z}_{\text{BEST}}\|_F,
\end{equation}
where $\mathcal{Z}_{\text{BEST}} \in \mathbb{R}^{n_1\times\ldots\times n_d}$ is the true best rank $\mathbf{r}$ approximation of $\mathcal{Z} \in \mathbb{R}^{n_1\times\ldots\times n_d}$. (For details, please see \cite{rauhut2017low} and the references therein.) For the rest of this section, we will always assume that $\mathcal{H}_{\mathbf{r}}$ is constructed in a way to satisfy \eqref{eq: hr_construction}. 

The following theorem is the main result of \cite{rauhut2017low}. 
It guarantees accurate reconstruction via TIHT guarantee when the measurement operator satisfies the TRIP$(\delta,3\mathbf{r})$ property for a sufficiently small $\delta$.
 Unfortunately, the condition \eqref{eqn: cant check this}, required by this theorem, is a bit stronger than 
\eqref{eq: hr_construction}, which is guaranteed to hold. As noted in \cite{rauhut2017low}, getting rid of the condition \eqref{eqn: cant check this} appears to be difficult if not impossible. That said,  \eqref{eq: hr_construction} is a worst-case estimate, and $\mathcal{H}_\mathbf{r}$ typically returns much better estimates. Moreover, numerical experiments show that the TIHT algorithm works well in practice, and indeed the condition \eqref{eqn: cant check this} does often hold, especially in early iterations of the algorithm. 
\begin{theorem}[\cite{rauhut2017low}, Theorem 1]\label{thm: rauhut_iht}
Let $0<a<1$, let $\mathcal{L}:\mathbb{R}^{n_1\times\ldots\times n_d}\rightarrow \mathbb{R}^{m}$ satisfy TRIP$(\delta,3\mathbf{r})$  with $\delta<a/4$ for some $a \in (0,1)$, and let  $\mathcal{X}^j$ and $\mathcal{Y}^j$ be defined as in \eqref{eq:IHT_Y}. 
We assume that 
\begin{equation}\label{eqn: cant check this}
    \|\mathcal{Y}^{j}-\mathcal{X}^{j+1}\|_F\leq (1+\xi_a)\|\mathcal{Y}^j-\mathcal{X}\|_F,\quad\text{ where }\quad \xi_a =\frac{a^2}{17(1+\sqrt{1+\delta_{3\mathbf{r}}})\|\mathcal{L}\|_{2\rightarrow 2}}.
\end{equation}
Suppose $\mathbf{y}=\mathcal{L}({\mathcal{X}})+\mathbf{e},$ where $\mathbf{e}\in\mathbb{R}^m$ is an arbitrary noise vector.
Then 
\begin{equation*}
    \|\mathcal{X}^{j+1}-\mathcal{X}\|_F\leq a^j\|\mathcal{X}^0-\mathcal{X}\|_F+\frac{b_a}{1-a}\|\mathbf{e}\|_2,
\end{equation*}
where $b_a = 2\sqrt{1+\delta}+\sqrt{4\xi_a + 2\xi_a^2}\|\mathcal{L}\|_{2\rightarrow2}.$ \end{theorem}

Theorem \ref{thm: rauhut_iht} shows that that low-rank tensor recovery is possible when the measurements satisfy the TRIP$(\delta,3\mathbf{r})$ property. In \cite{rauhut2017low}, the authors also show it is possible to randomly construct maps which satisfy this property with high probability. Unfortunately, these maps require first vectorizing the input tensor into a $n^d$-dimensional vector and then multiplying by an $m\times n^d$ matrix. This greatly limits the practical use of such maps since  this matrix  requires more memory than the original tensor. Thus, our results here for modewise TRIP are especially important and applicable in the tensor recovery setting.
The following corollary, which shows that we may choose $\mathcal{L}=\mathcal{A}$ or $\mathcal{A}_{2nd}$ (as in \eqref{eqn: A} or \eqref{a final}),  now follows immediately from combining  Theorem~\ref{thm: rauhut_iht} with Theorems \ref{thm: first compression} and \ref{thm: second compression}.

\begin{corollary}\label{cor: TIHT}
Assume the operator $\mathcal{L}$, is defined in one of the  following ways: \begin{enumerate}
\item[(a)]   $\mathcal{L} = \vect\circ\mathcal{A}\circ\mathcal{R}$, where $\mathcal{A}$ is defined as per \eqref{eqn: A} and the matrices $A_i$ satisfy the RIP$\left(\eps, \mathcal{S}_{1,2} \right)$, and $\delta = 4d'(3r)^d\eps < a/4$.
\item[(b)]   $\mathcal{L} = \mathcal{A}_{2nd}$ defined as in \eqref{a final}, its component matrices $A_i$ satisfy RIP$\left(\eps, \mathcal{S}_{1,2} \right)$ property, $\delta = 12d'^2(3r)^d \eps$, and $A_{\text{final}}$ satisfies the $RIP(\delta/3, \mathcal{B}_{1+\eps,\eps,1-\delta/3,\mathbf{r}})$ property.
\end{enumerate}
Consider the recovery problem from the noisy measurements $\mathbf{y}=\mathcal{L}(\mathcal{X})+\mathbf{e},$ where $\mathbf{e}\in\mathbb{R}^m$ is an arbitrary noise vector.
Let $0<a<1,$ and let  $\mathcal{X}^j$, and $\mathcal{Y}^j$ be defined as in \eqref{eq:IHT_Y}, and assume that \eqref{eqn: cant check this} holds. Then,  
\begin{equation*}
    \|\mathcal{X}^{j+1}-\mathcal{X}\|_F\leq a^j\|\mathcal{X}^0-\mathcal{X}\|_F+\frac{b_a}{1-a}\|\mathbf{e}\|,
\end{equation*}
where $b_a =  2\sqrt{1+\delta}+\sqrt{4\xi_a + 2\xi_a^2}\|\mathcal{A}\|_{2\rightarrow2}.$
\end{corollary}

\subsection{Experiments}\label{sec: experiments}
In this section, we show that TIHT can be used with modewise measurement maps in order to successfully reconstruct low-rank tensors. Herein we present numerical results for recovery of random four-mode tensors in $\R^{10 \times 10\times 10 \times 10}$ from both modewise Gaussian and SORS measurements. 

We run tensor iterative hard thresholding algorithm as defined in \eqref{eq:IHT_Y} to recover low-rank tensors from $m_0$ measurements for a variety of $m_0$ values. We compare the percentage of successfully recovered tensors from a batch of $100$ randomly generated low-rank tensors, as well as the average number of iterations used for recovery on the successful runs. We call a recovery process successful if the initial error between the true low-rank tensor and its initial random approximation decreases by a factor of at least $0.001$ in at most $1000$ iterations.  We compare standard vectorized measurements with proposed $2$-step partially modewise measurements that reshape a four-mode tensor into a $100 \times 100$ matrix, perform modewise measurements reducing each of the two reshaped modes to $m =90$, $80$ and $70$, and then vectorize that result and compress it further to the target dimension $m_0$. Herein we consider a variety of intermediate dimensions to demonstrate the stability of advantage of the modewise measurements over the vectorized ones. 

In addition to the smaller memory required for storing modewise measurement matrices, we show (Figure~\ref{fractions-recovered}) that modewise measurements are able to recover tensors from at least as small of a compressed representation as standard vectorized measurements can. Indeed, in the SORS case, modewise measurements can actually successfully recover low-rank tensors using a much smaller number of measurements (see Figure~\ref{fractions-recovered}, right column).  In Figure~\ref{gaussian-performance} we show that the described memory advantages do not result in the need for a substantially increased number of iterations in order to achieve our convergence criteria.  In the Gaussian case, the number of iterations needed when using a modewise measurements is at most twice the number as when using vectorized measurements. In the SORS case, modewise measurements actually require fewer iterations. Thus, modewise measurements are an effective, memory-efficient method of dimension reduction. 

\begin{figure}[ht]
  \centering
\includegraphics[width=.8\textwidth]{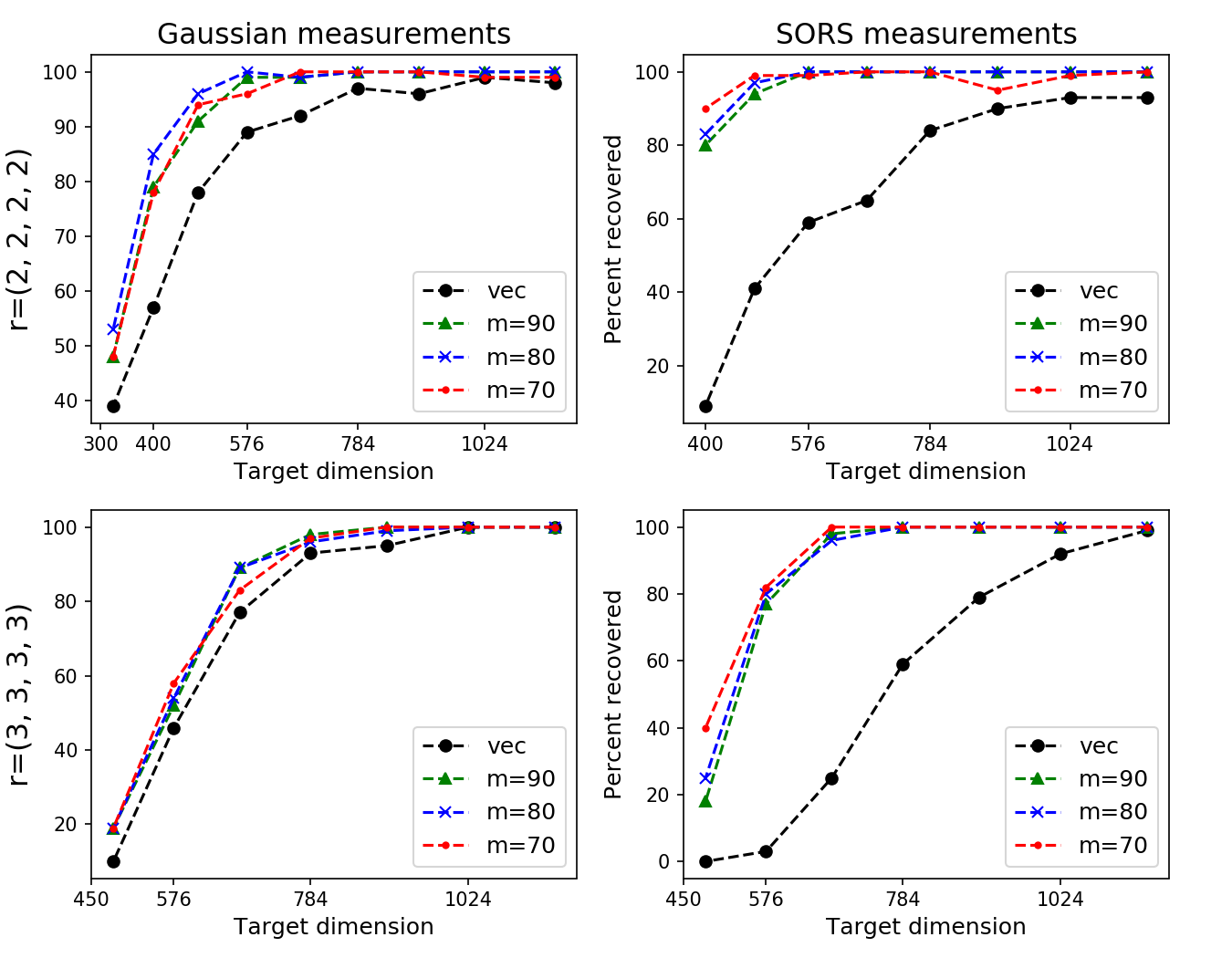}
\caption{
Fraction of successfully recovered random tensors out of a random sample of $100$ tensors in $\R^{10\times 10\times 10 \times 10}$ with various intermediate dimensions. A run is considered successful if it converged $0.001$ times the initial approximation error in at most $1000$ iterations.
}
\label{fractions-recovered}
\end{figure}

\begin{figure}[ht]
  \centering
\includegraphics[width=.8\textwidth]{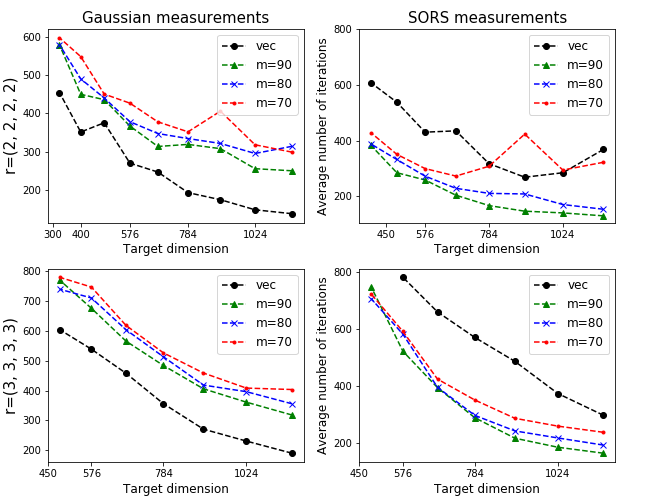}
\caption{
Average number of iterations until convergence among the successful runs. A run is considered successful if it converged $0.001$ times the initial approximation error in at most $1000$ iterations.
}
\label{gaussian-performance}
\end{figure}

\section{Proofs and theoretical guarantees}\label{sec: main proofs}
In this section, we will state auxiliary results that link the RIP property on a set $\mathcal{S}$ with the covering number of $\mathcal{S}$, and establish the covering number estimates for the subsets of interest.

\subsection{Auxiliary Results: RIP estimates}

For a set $\mathcal{S}$ we let  $\Net(\mathcal{S}, t)$ denote its covering number, i.e., the minimal cardinality of a net contained in $S$, such that every element of $S$ is within distance $t$ of an element of the net. For further discussion of the covering numbers, please see \cite{vershynin2018high}. 
The following proposition shows the estimates on the covering number of $\mathcal{S}$ can be used to show that maps constructed from sub-gaussian matrices have the RIP$(\eps, \mathcal{S})$ property.  Its proof, which is a generalization of the proof of \cite[Theorem 2]{rauhut2017low}, can be found in Appendix~\ref{sec: The proof of the Proposition}. 

\begin{proposition}\label{thm: sup_chaos}
Suppose $A \in \R^{m\times n^{\kappa}}$ has i.i.d. sub-gaussian entries. Let $\S\subseteq\R^{n \times \ldots\times n}$  be a subset of unit norm $\kappa$-mode tensors 
and let $\Net(\S, t)$ denote the covering number 
of $\S$. Then for any $0<\eta, \eps<1$ and
\begin{equation}\label{m-bound-prop}
m\geq \tilde{C}\eps^{-2}\max\left\{\left(\int_0^1\sqrt{\ln{\mathcal{N}(\S, t)}}dt\right)^2,1,\ln(\eta^{-1})\right\},
\end{equation}
for some suitably chosen constant $\tilde{C} >0$, with probability at least $1-\eta$, the map $\mathcal{A}(\mathcal{X})=A(\normalfont{\text{vect}}(\mathcal{X}))$ has the  RIP$(\eps, \mathcal{S})$ property, i.e.,
\begin{equation*}
    (1-\eps)\|\mathcal{X}\|^2
    \leq\|A(\text{\normalfont{vect}}(\mathcal{X}))\|^2    \leq (1+\eps)\|\mathcal{X}\|^2 \quad \text{ for all } \mathcal{X} \in \mathcal{S}.
\end{equation*}
\end{proposition}

We also need an analogue of Proposition~\ref{thm: sup_chaos} that holds for SORS matrices. Such results are known in the literature, however all of them have additional logarithmic terms compared to the i.i.d. sub-gaussian case.  We shall use the following result from \cite{mark-s-paper} which is a refinement of  Theorem 3.3 of \cite{oymak2018isometric}.

\begin{theorem}[Theorem 9 of  \cite{mark-s-paper}]\label{thm: sup_chaos_sors}
Suppose $A \in \R^{m\times n^{\kappa}}$ is a SORS matrix as per Definition~\ref{SORSdef} with $\Delta \leq C'$ for an absolute constant $C'$.  Let $\S$ 
be a subset of $\R^{n^\kappa}$, and let 
$\omega$ denote the Gaussian width (see, e.g. \cite{vershynin2018high}) of the projection of $\S$ onto the unit ball, $\left\{ {\bf x} / \| {\bf x} \|_2 ~\big|~{\bf x}\in\mathcal{S} \setminus \{ {\bf 0} \} \right\}$.
Let  $0<\eta, \eps<1$ and assume
\begin{equation}
m\geq \tilde{C}\eps^{-2} \omega^2 \ln^2\left[c_1 \omega^2 \ln(2 \eta^{-1})\eps^{-2}\right] \ln (2 \eta^{-1})\ln(2 en^{\kappa}\eta^{-1}),
\end{equation}
for some suitably chosen constants $\tilde{C}, c_1 >0$. Then, with probability at least $1-\eta$, the matrix $A$ has the  RIP$(\eps, \mathcal{S})$ property, i.e.,
\begin{equation*}
    (1-\eps)\|{\bf x}\|^2
    \leq\|A {\bf x}\|^2    \leq (1+\eps)\| {\bf x} \|^2 \quad \text{ holds for all } {\bf x} \in \mathcal{S}.
\end{equation*}
\end{theorem}

\begin{remark}\label{Gaussian_width_int}
It is known (see, e.g., Theorem 8.1.10 of \cite{vershynin2018high}) that the Gaussian width $\omega(\S)$ can be estimated by the same Dudley-type integral as used in \eqref{m-bound-prop}. Namely, for any set $\S$,
$$
\omega(\S) \le \int_0^\infty\sqrt{\ln{\mathcal{N}(\S, t)}}~dt.
$$
Additionally, if $\mathcal{S}$ is a subset of the unit ball, we have $\ln(\mathcal{N}(\S,t))=0$ for $t > 2$, and therefore,
$$
\omega(\S) \le \int_0^2\sqrt{\ln{\mathcal{N}(\S, t)}}~dt.
$$
\end{remark}

\subsection{Auxiliary results: covering estimates} The proofs of Corollaries \ref{cor:sub-gaussian model} and \ref{thm: second compression Gaussian} rely on applying  Proposition~\ref{thm: sup_chaos} to the sets $\mathcal{S}_{1,2}$ and $\mathcal{B}_{\mathbf{r},R,\theta,\mu}$ defined in Definitions \ref{def: S} and \ref{def: fancy B}. The proofs of Corollaries~\ref{cor:sors} and \ref{thm: second compression Fourier} analogously follow from an application of Theorem~\ref{thm: sup_chaos_sors}.
The following two lemmas provide covering estimates for these sets. Their proofs can be found in Appendix \ref{sec: proof of covering lemmas}. 
\begin{lemma}[Covering number for very low rank tensors]\label{lemma: cover}
 The covering number for the set $\S_{1,2}$ defined in Definition~\ref{def: S} satisfies $$\mathcal{N}(\S_{1,2},t) \le \left(\left(\frac{6\kappa}{t}\right)^{\kappa n} + 1\right)^2.$$
\end{lemma}

\begin{lemma}\label{lem: CoverfancyB1}
For all $\theta\geq 0,$ $0<\eps,\mu<1,$ $R\geq 1,$ and all $\mathbf{r} = (r, r, \ldots, r) \in \R^d$, the set  $\mathcal{B}_{R,\mu,\theta,\mathbf{r}} \subset \R^{n}$ defined in Definition \ref{def: fancy B} admits a covering with
\begin{equation}\label{eqn: generalize holger cover}
\mathcal{N}(\mathcal{B}_{R,\mu,\theta,\mathbf{r}},\|\cdot\|_F,\epsilon)\leq \left(\frac{6(d+1)}{\eps}\right)^{r^d + rnd} \left(R^2+\mu r\right)^{r^dd/2}\left(R^2+\mu r^d\right)^{dnr/2}R^{(d-1)dnr}.
\end{equation}
(Note that the right-hand  side is independent of $\theta$.)

\end{lemma}
\noindent 
\begin{remark}
If we set $\mu=\theta=0$ and $R=1$ we may obtain the covering number bound
\begin{equation*}
    \mathcal{N}(\mathcal{B}_{1,0,0,\mathbf{r}},\|\cdot\|_F,\epsilon)\leq \left(\frac{3(d+1)}{\epsilon}\right)^{r^d+dnr},
\end{equation*}
via a trival modification of the proof of Lemma~\ref{lem: CoverfancyB1} which doesn't require an application of Lemma~\ref{lem:setsetcover} when $\theta=0$.  This is the same as the estimate obtained in \cite[Lemma 5]{rauhut2017low} for $\mathcal{B}_{1,0,0,\mathbf{r}}$.
\end{remark}

\subsection{Proof of Theorem~\ref{thm: first compression} and Corollaries ~\ref{cor:sub-gaussian model} and \ref{cor:sors}}
\label{sec:ProofofTHm1}
 
In order to prove  Theorem~\ref{thm: first compression}, it will be useful to write $\mathcal{A}$ as a composition of maps \begin{equation}
    \mathcal{A}(\mathcal{Y})=\mathcal{A}_{d'}(\ldots(\mathcal{A}_1(\mathcal{Y}))), \text{ where } \mathcal{A}_i(\mathcal{Y})= \mathcal{Y}\times_i A_i \text{ for } 1\leq i \leq d'.
\end{equation}
Our argument will be based on showing that $\mathcal{A}_{i}$ approximately preserves the norm of $\mathcal{A}_{i-1}(\ldots(\mathcal{A}_1(\accentset{\circ}{\mathcal{X}})))$ for all $1\leq i \leq d'$. We first note that by \eqref{eqn: reshape HOSVD}, we may still write $\accentset{\circ}{\mathcal{X}}$ as a sum of $r^d$ orthogonal tensors.  This motivates Lemma~\ref{lem: orthogonal sums} which shows that if a linear operator $L$ on an inner product space $V$ satisfies certain assumptions, then it approximately preserves the norm of orthogonal sums (up to a factor depending on the number of terms).  
Lemma~\ref{lem: A1rank1} then provides sufficient conditions for the assumptions  of Lemma \ref{lem: orthogonal sums} to hold. Lastly, Lemma~\ref{lem: Ytform} will show that the image of the first $i-1$ compressions, $\mathcal{A}_{i-1}(\ldots(\mathcal{A}_1(\accentset{\circ}{\mathcal{X}}))),$  satisfies these conditions and  therefore that we may proceed inductively. The proofs of Lemmas \ref{lem: orthogonal sums}, \ref{lem: A1rank1}, and \ref{lem: Ytform} are deferred to  Appendix \ref{sec: The proof of orthogonal sums}. 
\begin{lemma}\label{lem: orthogonal sums}
Let $\mathcal{V}$ be an inner product space and let $\mathcal{L}$ be a linear operator on $\mathcal{V}.$ Let $\mathcal{U} \subset \mathcal{V}$ be a subspace of $\mathcal{V}$ spanned by an orthonormal system $\{{\bf v}_1,\ldots, {\bf v}_K\} \in \mathcal{V}$. 
Suppose that 
\begin{equation}\label{as: JLvsp}
    (1-\epsilon)\|{\bf v}_i\|^2\leq \|\mathcal{L}{\bf v}_i\|^2\leq (1+\epsilon)\|{\bf v}_i\|^2\quad\text{for all }1\leq i\leq K.
\end{equation}
and also that 
\begin{equation}\label{as: JLvsp2dprime}
    (1-\epsilon)\|{\bf v}_i\pm {\bf v}_j\|^2\leq \|\mathcal{L}({\bf v}_i\pm {\bf v}_j)\|^2\leq (1+\epsilon)\| {\bf v}_i\pm {\bf v}_j\|^2\quad\text{for all }1\leq i,j\leq K.
\end{equation}
Then we have 
\begin{equation*}
    (1-K\epsilon)\| {\bf w} \|^2\leq \|\mathcal{L} {\bf w}\|^2\leq (1+K\epsilon)\| {\bf w}\|^2 \text{ for all } {\bf w} \in {\mathcal U}.
\end{equation*}
\end{lemma}

 The next lemma checks that, if $A_{i_0}$ satisfies RIP$(\eps, \mathcal{S}_{1,2})$ property for some $1\leq i_0\leq d',$ then the operator $\mathcal{A}_{i_0}$ satisfies the conditions of Lemma~\ref{lem: orthogonal sums} for the system of rank one component tensors that are produced by our reshaping procedure.

\begin{lemma}\label{lem: A1rank1} Let $\{\mathcal{V}_1, \ldots, \mathcal{V}_{K}\} \in \R^{n \times \ldots \times n}$  be an orthonormal system of rank one tensors of the form $\mathcal{V}_k = \outprod_{i=1}^{d'} {\bf v}_{k}^i$ where $\|{\bf v}_{k}^i\| = 1$ for all $1\leq i \leq d'$. Let  $1\leq i_0 \leq d'$, suppose $A_{i_0}$ has the RIP$(\eps/2, \mathcal{S}_{1,2})$ property and assume that each  ${\bf v}_{k}^{i_0} $ is an element of the set $S_1$ defined in Definition \ref{def: S}. Then the conditions \eqref{as: JLvsp} and \eqref{as: JLvsp2dprime} are satisfied for (the vectorizations of) these $\{\mathcal{V}_i\}_{i=1}^K$ and $\mathcal{L} = \mathcal{A}_{i_0}$ defined via $\mathcal{A}_{i_0}(\mathcal{X})= \mathcal{X} \times_{i_0} A_{i_0}$.   
\end{lemma}

The next auxiliary lemma gives a formula for the tensor $\mathcal{Y}_t$ obtained by applying the first $t$ of the maps  $\mathcal{A}_i$. In particular, it shows that $\mathcal{Y}_t$  can be written as an orthogonal linear combination of $r^{\kappa(d'-t)}$ rank-one tensors of unit norm. Moreover, for each of the terms in this sum, the $(t+1)$-st component vector is ${\bf \accentset{\circ}{u}}_{j_{t+1}}^{t+1}$ as defined in \eqref{eqn: reshape HOSVD} and therefore is an element of the set $S_1$.

\begin{lemma}\label{lem: Ytform}
Let $\mathcal{Y}_0=\accentset{\circ}{\mathcal{X}}$ and 
$
\mathcal{Y}_t \coloneqq \mathcal{A}_t(\mathcal{Y}_{t-1})= \mathcal{Y}_{t-1}\times_t A_t
$ for all $t = 1, \ldots, d'.$ Then, for each $1\leq t\leq d'-1$, 
we may write 
\begin{equation}\label{eqn: Ytform}
\mathcal{Y}_t = \sum_{j_{d'}=1}^{r^\kappa}\ldots\sum_{j_{t+1}=1}^{r^{\kappa}} {\mathcal C}_t(j_{t+1},\ldots,j_{d'})\left[ \left(\outprod_{i=1}^t {\bf v}^i_{j_{t+1},\ldots,j_{d'}}\right) \outprod \left(\outprod_{i=t+1}^{d'}{\bf \accentset{\circ}{u}}_{j_i}^i\right)\right],
\end{equation}
where $\|{\bf v}^i_{j_{t+1},\ldots,j_{d'}}\|=1$ for all valid index subsets. (We note that the vectors ${\bf v}^i_{j_{t+1},\ldots,j_{d'}}$ implicitly depend on $t$. However, we suppress this dependence in order to avoid  cumbersome notation.) 
\end{lemma}
We are now ready to prove Theorem~\ref{thm: first compression}.
\begin{proof}[Proof of Theorem~\ref{thm: first compression}]
First, note that we can write $\mathcal{Y}_0=\accentset{\circ}{\mathcal{X}}$ as an orthogonal linear combination of $r^d$ norm one terms of the form 
\begin{equation*}
    \outprod_{i=1}^{d'}{\bf \accentset{\circ}{u}}^i_{j_i},\quad 1\leq j_i\leq r^\kappa,
\end{equation*}
where each of the vectors $\accentset{\circ}{\bf u}_{j_i}^i,\:1\leq j_i\leq r^\kappa,$ are obtained as the vectorization of a rank-one $\kappa$-mode tensor.
Therefore,
since $A_1$ satisfies RIP$(\eps, \mathcal{S}_{1,2})$, Lemma \ref{lem: A1rank1} allows us to apply Lemma \ref{lem: orthogonal sums} to see
\begin{equation}\label{eqn: Y1}
    \|\mathcal{A}_1(\accentset{\circ}{\mathcal{X}})\|\leq (1+2r^d\epsilon)\|\accentset{\circ}{\mathcal{X}}\|.
\end{equation}
Next, we apply Lemma \ref{lem: Ytform} and note that there are $r^{\kappa(d'-t)}$ terms appearing in the sum in \eqref{eqn: Ytform}. 
Therefore, Lemmas \ref{lem: orthogonal sums} and \ref{lem: A1rank1}  allow us to see that  
\begin{equation}\label{eqn: Yt+1}
    \|\mathcal{Y}_{t+1}\|\leq \left(1+2r^{\kappa(d'-t)}\epsilon\right) \|\mathcal{Y}_{t}\|
\end{equation}
for $1\leq t\leq d'-1$. Since $\mathcal{Y}_{d'}=\mathcal{A}(\accentset{\circ}{\mathcal{X}}),$ combining \eqref{eqn: Y1} and \eqref{eqn: Yt+1}
implies that the operator $\mathcal{A}$ defined in \eqref{eqn: A} satisfies
\begin{equation*}
    \|\mathcal{A}(\accentset{\circ}{\mathcal{X}})\|\leq\prod_{t=0}^{d'-1} \left(1+2r^{\kappa(d'-t)}\epsilon \right)\|\accentset{\circ}{\mathcal{X}}\|.
\end{equation*}
To complete the upper bound set $\alpha \coloneqq 2r^d\epsilon$ and note that $2\alpha<1$. Then, since $r\geq 2$ 
\begin{equation*}
\begin{aligned}
    \prod_{t=0}^{d'-1} \left(1+2r^{\kappa(d'-t)}\epsilon\right)&=\prod_{t=0}^{d'-1} (1+\alpha r^{-t\kappa})\\&=  
    1 + \alpha \sum_{t = 0}^{d' - 1} r^{-t \kappa} + \alpha^2 \sum_{\substack{t_1, t_2 = 0: \\ t_1 < t_2}}^{d' - 1} r^{-(t_1 + t_2) \kappa} + \ldots + \alpha^{d'}r^{-(1 + \ldots+ (d'-1))\kappa}\\
    &\le 1 + \alpha \sum_{t = 0}^{d' - 1} r^{-t \kappa} + \left(\alpha \sum_{t = 0}^{d' - 1} r^{-t \kappa} \right)^2 + \ldots + \left(\alpha \sum_{t = 0}^{d' - 1} r^{-t \kappa} \right)^{d'}
    \\
    &\le 1 + \alpha \sum_{t = 0}^{\infty} 2^{-t } + \left(\alpha \sum_{t = 0}^{\infty} 2^{-t } \right)^2 + \ldots + \left(\alpha \sum_{t = 0}^{\infty} 2^{-t } \right)^{d'}
    \\
    &\le 1 + 2\alpha  + (2\alpha )^2 + \ldots + (2\alpha )^{d'}
    \\
    &\le 1+2d'\alpha\\
    &= 1 + 4d'r^d \eps
\end{aligned}
\end{equation*}
which completes the proof of the upper bound.  The proof of the  lower bound is nearly identical.
\end{proof}

We will now prove  Corollaries~\ref{cor:sub-gaussian model} and \ref{cor:sors}.

\begin{proof}[Proof of Corollary~\ref{cor:sub-gaussian model}]
 We first note that  Lemma~\ref{lemma: cover} implies that
the integral from \eqref{m-bound-prop} can be bounded as
\begin{equation}\label{eqn: covering integral bound}
\int_0^1\sqrt{\ln{\mathcal{N}(\mathcal{S}_{1,2}, t)}}~dt \le
C\int_0^1\sqrt{\kappa n\ln(6\kappa/t)} ~dt \le
C\sqrt{\kappa n \ln(\kappa)}.
\end{equation}
Since  the set $\S_{1,2}$ contains (reshaped) unit norm $\kappa$-tensors, the assumption that $m$ satisfies \eqref{m_first_bound} implies that each of the $A_i$ will satisfy the assumptions of Proposition~\ref{thm: sup_chaos} with
$\eta/d'$ in place of $\eta$ and
$\eps^{-2} = (d/\kappa)^2r^{2d}\delta^{-2}$. Therefore, by the union bound, we have that all of the $A_i$ will satisfy RIP$(\eps, \mathcal{S}_{1,2})$ with probability  at least $1-\eta$,  and so the result now follows from Theorem~\ref{thm: first compression}.
\end{proof}

\begin{proof}[Proof of Corollary~\ref{cor:sors}]  

To estimate the Gaussian width $\S_{1,2}$, we use \eqref{eqn: covering integral bound}, Lemma \ref{lemma: cover}, and  Remark~\ref{Gaussian_width_int} to see that
\begin{equation*}
\begin{aligned}
\omega(\S_{1,2}) &\le \int_0^2\sqrt{\ln{\mathcal{N}(\S_{1,2}, t)}}~dt \le \int_0^1\sqrt{\ln{\mathcal{N}(\S_{1,2}, t)}}~dt + \sqrt{\ln{\mathcal{N}(\S_{1,2}, 1)}}\\
&\le C''\sqrt{\kappa n \ln(\kappa)} + C'''\sqrt{\kappa n \ln(\kappa)} = C\sqrt{\kappa n \ln(\kappa)}.
\end{aligned}
\end{equation*}
We now observe that the assumption \eqref{m_first_bound_sors} allows us to apply Theorem \ref{thm: sup_chaos_sors} with $\eta/d'$ in place of $\eta$ and $\eps^{-2} = (d/\kappa)^2 r^{2d}\delta^{-2}$. Therefore, analogous to the proof of Corollary \ref{cor:sub-gaussian model}, we conclude the proof of by taking the union bound and applying Theorem~\ref{thm: first compression}. 
\end{proof}
\subsection{Proof of Theorem~\ref{thm: second compression} and Corollaries~\ref{thm: second compression Gaussian} and \ref{thm: second compression Fourier}}
\label{sec: proof of second compression}

The key to proving Theorem~\ref{thm: second compression} is Lemma \ref{lem: AX in B}, which shows that the output of the first compression step $\mathcal{A}(\mathcal{R}(\mathcal{X}))$ lies in a set of nearly orthogonal tensors introduced in Definition \ref{def: fancy B}, and Lemma~\ref{lem: CoverfancyB1} which bounds the covering numbers for such tensors.  We can then get TRIP by applying Proposition \ref{thm: sup_chaos} to the vectorization of   $\mathcal{A}(\mathcal{R}(\mathcal{X}))$.

\begin{lemma}\label{lem: AX in B}
Let $\X$ be a unit norm d-mode tensor with HOSVD rank at most $\mathbf{r}=(r,\ldots,r)$. Let $\mathcal{A}$ be defined as in \eqref{eqn: A},  assume that the  matrices $A_i$ have the $RIP(\eps, \S_{1,2})$ property for all $1\leq i \leq d'$, and that $\delta = 12d'r^d \eps<1$. 
Then, the $d'$-mode tensor  $\mathcal{A}(\accentset{\circ}{\mathcal{X}})\in\mathbb{R}^{m\times\ldots\times  m}$
is an element of the set $\mathcal{B}_{1+\epsilon,\epsilon,1-\delta/3,{\bf r}'}$ (as per Definition~\ref{def: fancy B}).
Additionally, let $\tilde{\mathcal{B}}_{\epsilon,\delta,r} \coloneqq \left\{ \frac{\mathcal{X}}{\| \mathcal{X} \|_{\rm F}} ~\bigg|~\mathcal{X} \in \mathcal{B}_{1+\epsilon,\epsilon,1-\delta/3,{\bf r}'} \right\}$ and suppose that $m \ge r^{d-1}$.  Then, 
        \begin{equation}\begin{aligned}
    \mathcal{N}\left( \tilde{\mathcal{B}}_{\epsilon,\delta,r},\|\cdot\|_{\rm F}, t \right)
    \leq \left(\frac{9(d/\kappa+1)}{t}\right)^{r^d + r^\kappa md/\kappa} \left((1+\eps)^2+\eps r^d\right)^{dmr^\kappa/\kappa}(1 + \eps)^{d^2mr^\kappa \kappa^{-2}} \nonumber
    \end{aligned}\end{equation}
holds for all $t > 0$, and $1 > \delta > 0$.
\end{lemma}
\begin{proof}
As in \eqref{eqn: reshape HOSVD} we set $\accentset{\circ}{\mathcal{X}} = \mathcal{R}(\mathcal{X})$, and write  \begin{equation*}
    \accentset{\circ}{\mathcal{X}} = \sum_{j_{d'}=1}^{r^{\kappa}}\ldots\sum_{j_1=1}^{r^{\kappa}} \accentset{\circ}{C}(j_1,\ldots,j_{d'}) \outprod_{\ell=1}^{d'}\accentset{\circ}{\mathbf{u}}^{\ell}_{j_{\ell}},
    \end{equation*}
where, for all $i$, the vectors  $\{\accentset{\circ}{\mathbf{u}}^{i}_{j_{i}}\}_{j_i=1}^{r^\kappa}$ are mutually orthogonal.  Recall that by 
 $\eqref{eqn: newfactors}$, we have
\begin{equation*}
    \mathcal{A}(\accentset{\circ}{\mathcal{X}}) = \sum_{j_{d'}=1}^{r^{\kappa}}\ldots\sum_{j_1=1}^{r^{\kappa}} \accentset{\circ}{C}(j_1,\ldots,j_{d'}) \left(A_1{\bf\accentset{\circ}{u}}^1_{j_1} \outprod \ldots \outprod A_{d'}{\bf \accentset{\circ}{u}}^{d'}_{j_{d'}} \right).
\end{equation*} 
By RIP$(\eps, \S_{1,2})$ property, $\|A_{i}\accentset{\circ}{\mathbf{u}}^{i}_{j_{i}}\|\leq (1+\epsilon)$
for all $1\leq i \leq d'$ and all $1\leq j_i\leq r^\kappa.$
Additionally, by Lemma \ref{lem: preserves angles} (stated in Appendix A) we have  
\begin{equation*}
   |\langle A_{i}{\bf\accentset{\circ}{u}}^{i}_{j_{i}},    A_{i}{\bf\accentset{\circ}{u}}^{i}_{j'_{i}} \rangle| < \epsilon
\end{equation*} 
for all $1\leq i\leq d'$ and all $1\leq j_i,j_i'\leq r^\kappa$ such that $j'_i\neq j_i$. The properties of the core tensor (c) and (d) are preserved under the action of $\mathcal{A}$ and are satisfied for a unit norm, tensor in the HOSVD standard form.  Finally, Theorem \ref{thm: first compression} implies that $\mathcal{A}$ satisfies the TRIP$(\delta/3,\mathbf{r})$ property which in turn guarantees that  $\mathcal{A}(\accentset{\circ}{\mathcal{X}})$ will also satisfy property (e) of Definition~\ref{def: fancy B} with $\theta = 1 - \delta/3$.

Applying Lemma~\ref{lem: CoverfancyB1} (with $R=(1+\eps)$ and  $m$, $d'$, $r^\kappa,$ and $t$ in place of $n,d,r,$ and $\epsilon$) and the assumption $m \ge r^{d-1}$ we obtain 
        \begin{equation}\begin{aligned}
    \mathcal{N}\left(\mathcal{B}_{1+\epsilon,\epsilon,1-\delta/3,{\bf r}'},\|\cdot\|_{\rm F},t \right)
    \leq \left(\frac{6(d/\kappa+1)}{t}\right)^{r^d + r^\kappa md/\kappa} \left((1+\eps)^2+\eps r^d\right)^{dmr^\kappa/\kappa}(1 + \eps)^{d^2mr^\kappa \kappa^{-2}}. \nonumber
    \end{aligned}\end{equation}
Furthermore, a geometric rescaling argument implies that
$$\mathcal{N}\left( \tilde{\mathcal{B}}_{\epsilon,\delta,r},\|\cdot\|_{\rm F}, t \right) = \mathcal{N}\left( \left\{ \frac{\mathcal{X}}{\| \mathcal{X} \|_{\rm F}} ~\bigg|~\mathcal{X} \in \mathcal{B}_{1+\epsilon,\epsilon,1-\delta/3,{\bf r}'} \right\},\|\cdot\|_{\rm F}, t \right) \leq  \mathcal{N}\left(\mathcal{B}_{1+\epsilon,\epsilon,1-\delta/3,{\bf r}'},\|\cdot\|_{\rm F}, 2t/3 \right)$$
hols for all $\delta \leq 1$.  The stated result now follows.
\end{proof}

Theorem~\ref{thm: second compression} now follows from a direct application of Lemma~\ref{lem: AX in B} and Theorem~\ref{thm: first compression}.

\begin{proof}[Proof of Theorem~\ref{thm: second compression}]
Theorem \ref{thm: first compression} implies that $\mathcal{A}$ satisfies the TRIP$(\delta/3,\mathbf{r})$ property and Lemma ~\ref{lem: AX in B} implies that $\mathcal{A}(\accentset{\circ}{\mathcal{X}})$ belongs to the set $\mathcal{B}_{1+\epsilon,\epsilon,1-\delta/3,{\bf r}'}$.
Therefore, we have 
    \begin{equation*}
     \|\mathcal{A}_{2nd}(\mathcal{X})\|^2\leq (1+\delta/3)\|\mathcal{A}(\accentset{\circ}{\mathcal{X}})\|\leq (1+\delta/3)^2\|\mathcal{X}\|^2\leq     (1+\delta)\|\mathcal{X}\|^2, 
\end{equation*}
and 
\begin{equation*}
 (1-\delta)\|\mathcal{X}\|^2    \leq (1-\delta/3)^2\|\mathcal{X}\|^2\leq  (1-\delta/3)\|\mathcal{A}(\accentset{\circ}{\mathcal{X}})\|\leq \|\mathcal{A}_{2nd}(\mathcal{X})\|^2. 
\end{equation*}
\end{proof}

The proofs of Corollaries~\ref{thm: second compression Gaussian} and \ref{thm: second compression Fourier} require an additional estimate bounding the Gaussian width of $\tilde{\mathcal{B}}_{\epsilon,\delta,r}$ from Lemma~\ref{lem: AX in B}.

\begin{remark}\label{dudley_b_estimate} Let $\tilde{\mathcal{B}}_{\epsilon,\delta,r} \subset \R^{m\times \ldots\times m}$ be the set of $d'$-mode tensors defined as per Lemma~\ref{lem: AX in B} and Definition~\ref{def: fancy B}.  We can see that
\begin{equation*}\begin{aligned}
        \int_0^1\sqrt{\ln{\mathcal{N}\left(\tilde{\mathcal{B}}_{\epsilon,\delta,r},t \right)}}~dt &\leq \int_0^2\sqrt{\ln{\mathcal{N}\left(\tilde{\mathcal{B}}_{\epsilon,\delta,r},t \right)}}~dt \leq \int_0^1\sqrt{\ln{\mathcal{N}\left(\tilde{\mathcal{B}}_{\epsilon,\delta,r},t \right)}}~dt + \sqrt{\ln{\mathcal{N}\left(\tilde{\mathcal{B}}_{\epsilon,\delta,r},1 \right)}}.
\end{aligned}\end{equation*}
If $m\geq r^{d-1},$ then we may apply Lemma~\ref{lem: AX in B} and use the concavity of $\sqrt{\cdot}$ to see that 
\begin{equation*}\begin{aligned}   \int_0^1\sqrt{\ln{\mathcal{N}\left(\tilde{\mathcal{B}}_{\epsilon,\delta,r},t \right)}}~dt
        \leq \int_0^1\sqrt{\left(r^d + \frac{r^\kappa md}{\kappa}\right)\ln\left(\frac{9(d/\kappa+1)}{t}\right)}~&dt \\
        \quad\quad\quad+ \int_0^1\sqrt{{\frac{dmr^\kappa}{\kappa}}\ln \left(\left(1+\epsilon\right)^2+\epsilon r^d\right)}~dt+\int_0^1&\sqrt{ {\frac{d^2mr^\kappa}{ \kappa^{2}}}\ln\left(1+\epsilon \right)}~dt\\
        \le C \sqrt{r^d+\frac{dmr^\kappa}{\kappa}}\sqrt{\ln\left(\frac{d}{\kappa}+1\right)} + \sqrt{{\frac{dmr^\kappa}{\kappa}}\ln \left((1+\epsilon)^2+\epsilon r^d\right)} + &\sqrt{ {\frac{d^2mr^\kappa}{ \kappa^{2}}}\ln\left(1+\epsilon \right)}.
\end{aligned}\end{equation*}
Furthermore, we also have
\begin{equation*}\begin{aligned}   \sqrt{\ln{\mathcal{N}\left(\tilde{\mathcal{B}}_{\epsilon,\delta,r},1 \right)}}
        \leq C' \sqrt{r^d+\frac{dmr^\kappa}{\kappa}}\sqrt{\ln\left(\frac{d}{\kappa}+1\right)}& \\
        \quad\quad\quad+ \sqrt{{\frac{dmr^\kappa}{\kappa}}\ln \left((1+\epsilon)^2+\epsilon r^d\right)} + &\sqrt{ {\frac{d^2mr^\kappa}{ \kappa^{2}}}\ln\left(1+\epsilon\right)}.
\end{aligned}\end{equation*}
\end{remark}

\begin{proof}[Proof of Corollary~\ref{thm: second compression Gaussian}] Similar to the proof of  Corollary~\ref{cor:sub-gaussian model}, the assumption that $m$ satisfies $\eqref{m_first_bound_two_step}$, 
 implies that with probability at least $1-\eta/2$, all of
the $A_i$ satisfy the RIP$\left(\eps, \mathcal{S}_{1,2} \right)$  property, 
$\delta = 12d'r^d \eps<1$.
The assumption \eqref{m_first_bound_two_step} also implies $m\geq r^{d-1}$. Therefore, by Proposition~\ref{thm: sup_chaos}, the estimate for the Dudley-type integral from Remark~\ref{dudley_b_estimate}, and the linearity of $A_{\text{2nd}} \circ {\rm vect}$ we 
see that $A_{\text{2nd}} \circ {\rm vect}$ satisfies the  RIP$\left(\delta/3,\mathcal{B}_{1+\epsilon,\epsilon,1-\delta/3,\mathbf{r'}} \right)$ property
with probability at least $1-\eta/2$ as long as
\begin{equation*}
m_\text{2nd}\geq C\delta^{-2}\max\left\{\left(r^d+\frac{dmr^\kappa}{\kappa} \right)\ln\left(\frac{d}{\kappa}+1\right) + {\frac{dmr^\kappa}{\kappa}}\ln \left(1+\delta r^d\right) + {\frac{d^2mr^\kappa\delta}{ \kappa^{2}}},\ln\left(\frac{2}{\eta}\right)\right\}.
\end{equation*}
The result now follows  by applying Theorem \ref{thm: second compression}.
\end{proof}

\begin{proof}[Proof of Corollary~\ref{thm: second compression Fourier}]
Repeating the arguments used in the proof of Corollary~\ref{cor:sors}, we see that the assumption that  $m$ satisfies $\eqref{eqn: m_big_first_fouri_ber}$, implies that with probability at least $1-\eta/2$, all of
the $A_i$ satisfy the RIP$\left(\eps, \mathcal{S}_{1,2} \right)$  property with 
$\delta = 12d'r^d \eps<1$. 
We also note that \eqref{eqn: m_big_first_fouri_ber} implies that $m\geq r^{d-1}$ so that we may apply Remark~\ref{dudley_b_estimate}. Thus, \eqref{eqn: m2_big_Fourier} and Theorem \ref{thm: sup_chaos_sors} imply that $A_{\text{2nd}}$ satisfies the  RIP$\left(\delta/3,\mathcal{B}_{1+\epsilon,\epsilon,1-\delta/3,\mathbf{r'}} \right)$ property with probability at least $1-\eta/2.$ Therefore, the result now follows from Theorem \ref{thm: second compression}. 
\end{proof}

\section{Conclusion and Future Work}\label{sec: conclusion}

In this paper, we have proved that several modewise linear maps (with sub-Gaussian and subsampled from the orthogonal ensemble -- e.g., discrete Fourier -- measurements) have the TRIP for tensors with low-rank HOSVD decompositions. Our measurements maps require significantly less memory than previous works such as \cite{rauhut2017low} and \cite{grotheer2019iterative} that establish TRIP for vectorized measurements. We also note that unlike other closely related works such as \cite{jin2019faster} and \cite{iwen2019lower} that establish modewise Johnson Lindenstrauss embeddings, our results hold for all low-HOSVD rank tensors whereas previous work focuses on finite sets or for tensors lying in a low-dimensional vector space.  In our experiments, we have demonstrated that we are able to recover low-rank tensors from a compressed representation produced via two-step modewise measurements. Moreover, we show that we are able to achieve such recovery from a lower compressed dimension than with purely vectorized measurements, establishing yet another advantage. 

A natural direction for future work would involve extending these results to other tensor formats including, e.g., tensors which admit compact tensor train, (hierarchical) Tucker, and/or CP decompositions instead.  Additional projects of value might include parallel implementations of the TIHT algorithm using modewise maps that fully leverage their structure, as well as more memory efficient TIHT variants which reconstruct the factors of a given low-rank tensor from its measurements instead of reconstructing the entire tensor in uncompressed form.  Indeed, such a memory efficient TIHT implementation in combination with using modewise measurements would allow for memory efficient low-rank tensor reconstruction from the measurement stage all the way through reconstruction of the final approximation in compressed form.

\appendix
\section{The Proof of Lemmas~\ref{lem: orthogonal sums}, \ref{lem: A1rank1}, and \ref{lem: Ytform}}\label{sec: The proof of orthogonal sums}
The proof of Lemma~\ref{lem: orthogonal sums} requires the following well-known auxiliary lemma. For completeness, we provide a short proof below.
\begin{lemma}\label{lem: preserves angles}
Let $\mathcal{V}$ be an inner product space and let $\mathcal{L}$ be a linear operator on $\mathcal{V}.$ Let $\{{\bf v}_1,\ldots, {\bf v}_K\}$ be a finite orthonormal system in $\mathcal{V}$ (that is, $\|{\bf v}_i\|=1$ for all $i$ and  $\langle {\bf v}_i, {\bf v}_j\rangle=0$ for all $i\neq j$). Suppose that 
\begin{equation*}
    (1-\epsilon)\|{\bf v}_i\pm {\bf v}_j\|^2\leq \|\mathcal{L}({\bf v}_i\pm {\bf v}_j)\|^2\leq (1+\epsilon)\|{\bf v}_i\pm {\bf v}_j\|^2\quad\text{for all }1\leq i,j\leq K,i\neq j.
\end{equation*}
Then \begin{equation*}
    |\langle \mathcal{L}{\bf v}_i,\mathcal{L}{\bf v}_j\rangle|\leq \epsilon\quad\text{for all }i\neq j.
\end{equation*}
\end{lemma}

\begin{proof}
 Let $i\neq j.$ Then,
 \begin{equation*}
\begin{aligned}
4\langle \mathcal{L}{\bf v}_i,\mathcal{L}{\bf v}_j\rangle &= \|\mathcal{L}({\bf v}_i+{\bf v}_j)\|^2-\|\mathcal{L}({\bf v}_i-{\bf v}_j)\|^2\\
&\leq (1+\epsilon) \| {\bf v}_i+{\bf v}_j\|^2 - (1-\epsilon)\| {\bf v}_i-{\bf v}_j\|^2\\
&= (1+\epsilon)(\| {\bf v}_i\|^2+\|{\bf v}_j\|^2+2\langle {\bf v}_i, {\bf v}_j\rangle)-
(1-\epsilon)(\|{\bf v}_i\|^2+\|{\bf v}_j\|^2-2\langle {\bf v}_i,{\bf v}_j\rangle)\\
&=4 \langle {\bf v}_i, {\bf v}_j\rangle+2\epsilon(\| {\bf v}_1\|^2+\| {\bf v}_2\|^2)\\
&=4\epsilon,
\end{aligned}
\end{equation*}
where the last inequality follows from the fact that $\|{\bf v}_1\|^2=\| {\bf v}_2\|^2=1$ and $\langle {\bf v}_i, {\bf v}_j\rangle = 0$. Thus, $\langle \mathcal{L} {\bf v}_i,\mathcal{L} {\bf v}_j\rangle \leq \epsilon.$ The reverse inequality is similar.
\end{proof}
We may now prove Lemma \ref{lem: orthogonal sums}.

\begin{proof}[The Proof of Lemma \ref{lem: orthogonal sums}]
We argue by induction on $K.$ When $K=1$, the result is immediate from  \eqref{as: JLvsp} and the fact that $\mathcal{L}$ is linear. Now assume the result is true for $K-1.$ 
An arbitrary element of $\mathcal{U}$ may be written as 
\begin{equation*}
    {\bf w}=\sum_{i=1}^Kc_i {\bf v}_i
\end{equation*}
where $c_1,\ldots,c_K$ are scalars. 
We will write  ${\bf w} = {\bf w}_{K-1}+c_K {\bf v}_K$, where 
\begin{equation*}
    {\bf w}_{K-1}\coloneqq \sum_{i=1}^{K-1}c_i{\bf v}_i.
\end{equation*}
By construction, we have
\begin{equation*}\begin{aligned}
    \|\mathcal{L}{\bf w}\|^2 = \|\mathcal{L}{\bf w}_{K-1}\|^2 + \|c_K {\mathcal L} {\bf v}_K\|^2 + 2c_K\langle \mathcal{L}{\bf w}_{K-1},\mathcal{L}{\bf v}_K\rangle.
\end{aligned}
\end{equation*}
We may use the inequality $2ab\leq a^2+b^2$ along with Lemma \ref{lem: preserves angles} to see

\begin{equation}\begin{aligned}
    2c_K\langle \mathcal{L}{\bf w}_{K-1},\mathcal{L}{\bf v}_K\rangle = \sum_{i=1}^{K-1}2c_ic_K\langle \mathcal{L}{\bf v}_i,\mathcal{L}{\bf v}_K\rangle&\leq\epsilon\sum_{i=1}^{K-1}2c_Kc_i\\ &\leq\epsilon\sum_{i=1}^{K-1}(c_K^2+c_i^2)\leq (K-1)\epsilon c_K^2 + \epsilon\sum_{i=1}^{K-1}c_i^2.
\end{aligned}\end{equation}
By the inductive assumption, 
\begin{equation*}
    \|\mathcal{L}{\bf w}_{K-1}\|^2 \leq (1+(K-1)\epsilon)\|{\bf w}_{K-1}\|^2 = (1+(K-1)\epsilon)\sum_{i=1}^{K-1}c_i^2.
\end{equation*}
Thus,
\begin{equation}\begin{aligned}
    \|\mathcal{L}{\bf w}\|^2 \leq (1+(K-1)\epsilon)\sum_{i=1}^{K-1}c_i^2 + (1+\epsilon)c_K^2 &+ (K-1)\epsilon c_K^2 + \epsilon\sum_{i=1}^{K-1}c_i^2\\
    &=(1+K\epsilon)\sum_{i=1}^Kc_i^2 =(1+K\epsilon)\| {\bf w} \|^2.
\end{aligned}\end{equation}
The reverse inequality is similar.
\end{proof}

\begin{proof}[Proof of Lemma~\ref{lem: A1rank1}]
Without loss of generality, we consider the case where $i_0=1$. By assumption, we have ${\bf v}^1_k\in S_1$ for all $1\leq k\leq K$.  Therefore, since $A_1$ is assumed to have the  RIP$(\eps/2, \S_{1,2})$ property, we have
\begin{equation*}
1-\epsilon/2 = (1-\epsilon/2)\|{\bf v}^{1}_k\|^2 \leq \|A_1{\bf v}^{1}_k \|^2 \leq (1+\epsilon/2)\|{\bf v}^{1}_k\|^2 = 1+\epsilon/2.
\end{equation*}
Now, \eqref{as: JLvsp} follows from  the fact that \begin{equation*}
\left\|\mathcal{A}_1 \left(\outprod_{i=1}^{d'}{\bf v}^{i}_k \right) \right\|=\|(A_1{\bf v}^{1}_k) \outprod {\bf v}^{2}_k\outprod\ldots \outprod {\bf v}^{d'}_k\|=\|A_1{\bf v}^{1}_k\|\| {\bf v}^{2}_k\|\ldots \|{\bf v}^{d'}_k\|=\|A_1{\bf v}^{1}_k\|.
\end{equation*}
and the fact that the $\mathbf{v}^i_k$  have norm one.
To prove \eqref{as: JLvsp2dprime}, we let $1\leq k_1,k_2\leq K$ and recall that $\mathcal{V}_{k_1} = \outprod_{i=1}^{d'}{\bf v}^{i}_{k_1}$ and $\mathcal{V}_{k_2} = \outprod_{i=1}^{d'}{\bf v}^{i}_{k_2}$, where each of the ${\bf v}^i_{k_1}$ and ${\bf v}^i_{k_2}$ have norm one. If $k_1=k_2,$ \eqref{as: JLvsp2dprime} follows immediately from \eqref{as: JLvsp}. Otherwise, we may use the assumption that $\{\mathcal{V}_1,\ldots,\mathcal{V}_K\}$ form an orthonormal system to see  \begin{equation*}0=\langle \mathcal{V}_{k_1},\mathcal{V}_{k_2}\rangle =\left\langle \outprod_{i=1}^{d'}{\bf v}^{i}_{k_1} , \outprod_{i=1}^{d'}{\bf v}^{i}_{k_2} \right\rangle =\prod_{i=1}^{d'} \langle {\bf v}^{i}_{k_1}, {\bf v}^{i}_{k_2} \rangle.
\end{equation*}
This implies that there exists an $i$ such that  $\langle {\bf v}^{i}_{k_1} , {\bf v}^{i}_{k_2} \rangle=0$.
If $\langle {\bf v}^{1}_{k_1} , {\bf v}^{1}_{k_2} \rangle=0$, then since $A_1$ satisfies RIP$(\eps, \S_{1,2})$, we may apply Lemma \ref{lem: preserves angles}, and the Cauchy-Schwarz inequality to see that 
\begin{equation*}
\begin{aligned}
    \left| \left \langle \mathcal{A}_1 \left(\outprod_{i=1}^{d'}{\bf v}^{i}_{k_1} \right),\mathcal{A}_1 \left(\outprod_{i=1}^{d'}{\bf v}^{i}_{k_2} \right) \right\rangle \right|&=|\langle A_1{\bf v}^{1}_{k_1}, A_1{\bf v}^{1}_{k_2} \rangle||\langle {\bf v}^{2}_{k_1} , {\bf v}^{2}_{k_2} \rangle|\ldots |\langle {\bf v}^{d'}_{k_1} , {\bf v}^{d'}_{k_2} \rangle|\\&\leq(\epsilon/2) \prod_{i=2}^{d'} \|{\bf v}^{i}_{k_1}\| \|{\bf v}^{i}_{k_2}\| \\&= \eps/2.
\end{aligned}
\end{equation*}
On the other hand, if $\langle {\bf v}^{\ell}_{k_1} , {\bf v}^{\ell}_{k_2} \rangle=0$ for some $\ell\neq 1$ then we have 
\begin{equation*}
  \left| \left \langle \mathcal{A}_1 \left(\outprod_{i=1}^{d'}{\bf v}^{i}_{k_1} \right),\mathcal{A}_1 \left(\outprod_{i=1}^{d'}{\bf v}^{i}_{k_2} \right) \right\rangle \right|=|\langle A_1{\bf v}^{1}_{k_2}, A_1{\bf v}^{1}_{k_2} \rangle||\langle {\bf v}^{2}_{k_1} , {\bf v}^{2}_{k_2} \rangle|\ldots |\langle {\bf v}^{d'}_{k_1} , {\bf v}^{d'}_{k_2} \rangle|=0.
\end{equation*}
Therefore, in either case we have
\begin{equation*}
\begin{aligned}
\left\|\mathcal{A}_1 \left(\outprod_{i=1}^{d'}{\bf v}^{i}_{k_1} + \outprod_{i=1}^{d'}{\bf v}^{i}_{k_2} \right) \right\|^2&= \left\|\mathcal{A}_1 \left( \outprod_{i=1}^{d'}{\bf v}^{i}_{k_1} \right) \right\|^2+ \left\|\mathcal{A}_1 \left(\outprod_{i=1}^{d'}{\bf v}^{i}_{k_2} \right) \right\|^2 + 2 \left\langle \mathcal{A}_1 \left(\outprod_{i=1}^{d'}{\bf v}^{i}_{k_2} \right), \mathcal{A}_1 \left(\outprod_{i=1}^{d'}{\bf v}^{i}_{k_2} \right)  \right\rangle \\
&\leq (1+\epsilon/2) \left( \left\|\outprod_{i=1}^{d'}{\bf v}^{i}_{k_1} \right\|^2+ \left\|\outprod_{i=1}^{d'}{\bf v}^{i}_{k_2}\right\|^2\right) +  \epsilon\\ 
&= (1+\epsilon/2) \left( \left\|\outprod_{i=1}^{d'}{\bf v}^{i}_{k_1} \right\|^2+ \left\|\outprod_{i=1}^{d'}{\bf v}^{i}_{k_2}\right\|^2\right) + (\epsilon/2) \left( \left\|\outprod_{i=1}^{d'}{\bf v}^{i}_{k_1} \right\|^2+ \left\|\outprod_{i=1}^{d'}{\bf v}^{i}_{k_2}\right\|^2\right)\\
&=(1+\epsilon) \left\|\outprod_{i=1}^{d'}{\bf v}^{i}_{k_1} + \outprod_{i=1}^{d'}{\bf v}^{i}_{k_2} \right\|^2,
\end{aligned}
\end{equation*}
where in the last equality, we used orthogonality to see
\begin{equation*}
    \left\|\outprod_{i=1}^{d'}{\bf v}^{i}_{k_1}+\outprod_{i=1}^{d'}{\bf v}^{i}_{k_2} \right\|^2= \left\|\outprod_{i=1}^{d'}{\bf v}^{i}_{k_1} \right\|^2+ \left\|\outprod_{i=1}^{d'}{\bf v}^{i}_{k_2} \right\|^2.
\end{equation*}
The reverse inequality is similar.
\end{proof}

\begin{proof}[Proof of Lemma~\ref{lem: Ytform}]
We argue by induction. When $t = 0$, the decomposition \eqref{eqn: Ytform} follows immediately from  \eqref{eqn: reshape HOSVD} with ${\mathcal C}_0 = \accentset{\circ}{\mathcal C}$ and the fact that  $\mathcal{Y}_0 = \accentset{\circ}{\X}$.

 Now suppose that the result is true for $t$ for some $0\leq t\leq d-2$. Then,
\begin{equation*}
\begin{aligned}
\mathcal{Y}_{t+1} &= \mathcal{Y}_{t} \times_{t+1} \mathcal{A}_{t+1} \\ &= \sum_{j_{d'}=1}^{r^\kappa}\ldots\sum_{j_{t+1}=1}^{r^{\kappa}} {\mathcal C}_t(j_{t+1},\ldots,j_{d'})\left[ \left(\outprod_{i=1}^t {\bf v}^i_{j_{t+1},\ldots,j_{d'}}\right) \outprod \left(A_{t+1}{\bf\accentset{\circ}{u}}_{j_{t+1}}^{t+1}\right) \outprod \left(\outprod_{i=t+2}^{d'}{\bf \accentset{\circ}{u}}_{j_i}^i\right)\right].
\end{aligned}
\end{equation*}
Therefore, summing over $j_{t+1}$-st mode and normalizing yields
\begin{equation*}
\begin{aligned}
\mathcal{Y}_{t+1}
&=\sum_{j_{d'}=1}^{r^\kappa}\ldots\sum_{j_{t+2}=1}^{r^{\kappa}} {\mathcal C}_{t+1}(j_{t+2},\ldots,j_{d'})\left[ \left(\outprod_{i=1}^{t+1} {\bf v}^i_{j_{t+2},\ldots,j_{d'}}\right) \outprod \left(\outprod_{i=t+2}^{d'}{\bf \accentset{\circ}{u}}_{j_i}^i\right)\right],
\end{aligned}
\end{equation*}
where new vectors ${\bf v}^i_{j_{t+2},\ldots,j_{d'}}$ (with one less subscript) are defined as
$$
{\bf v}^i_{j_{t+2},\ldots,j_{d'}} = \frac{{\bf \tilde  v}^i_{j_{t+2},\ldots,j_{d'}}}{\left\|{\bf \tilde  v}^i_{j_{t+2},\ldots,j_{d'}}\right\|} \quad \text{ and }   \quad {\mathcal C}_{t+1}(j_{t+2},\ldots,j_{d'})= \left\| {\bf \tilde v}^i_{j_{t+2},\ldots,j_{d'}}\right\|,
$$
where
$${\bf \tilde  v}^i_{j_{t+2},\ldots,j_{d'}} = \sum_{j_{t+1} = 1}^{r^{\kappa}}  {\mathcal C}_t(j_{t+1},\ldots,j_{d'}) \left(\outprod_{i=1}^t {\bf v}^i_{j_{t+1},\ldots,j_{d'}}\right) \outprod A_{t+1}{\bf\accentset{\circ}{u}}_{j_{t+1}}^{t+1}.
$$

\end{proof}

\section{The proof of Lemmas \ref{lemma: cover} and \ref{lem: CoverfancyB1}}\label{sec: proof of covering lemmas}

 \begin{proof}[The proof of Lemma \ref{lemma: cover}]
Classical results \cite{vershynin2018high} utilizing volumetric estimates show that 
\begin{equation}\label{eqn: classic nets}
\mathcal{N}(\mathbb{S}^{n-1},\|\cdot\|_2,t)\leq \left(\frac{3}{t}\right)^{n}.
\end{equation}
Let $\mathcal{N}_1$ be the $k$-fold Kronecker product of $t/2\kappa$-nets for $\mathbb{S}^{n-1}$. By \eqref{eqn: classic nets}, we may choose $\mathcal{N}_1$ to have cardinality at most $(6\kappa/t)^{\kappa n}$. Moreover, for any $\X=\bigotimes _{i=1}^\kappa  u^i \in S_1$ (defined by \eqref{s1}),
\begin{equation*}
\begin{aligned}
 \inf_{\tilde\X \in \mathcal{N}_1}\|\tilde\X -\X\|_{\rm F} &\le \inf_{\substack{\tilde\X \in \mathcal{N}_1\\ \tilde\X=\otimes _{i=1}^\kappa \tilde u^i}}
 \left\|\bigotimes_{i=1}^\kappa u^i - \bigotimes_{i=1}^\kappa \tilde u^i\right\|_{\rm F} \\ &\le \inf_{\substack{\tilde\X \in \mathcal{N}_1\\ \tilde\X=\otimes _{i=1}^\kappa \tilde u^i}}\sum_{j = 1}^{\kappa} \left\|\bigotimes_{i=1}^{j-1} u^i \bigotimes_{i=j}^\kappa \tilde u^i- \bigotimes_{i=1}^j u^i \bigotimes_{i=j+1}^\kappa \tilde u^i\right\|_{\rm F} \\
 &=\inf_{\substack{\tilde\X \in \mathcal{N}_1\\ \tilde\X=\otimes _{i=1}^\kappa \tilde u^i}} \sum_{j = 1}^{\kappa}\left(\prod_{i = 1}^{j-1}\|\tilde u^i\|_2\right) \|\tilde u^j - u^j\|_2\left(\prod_{i = j+1}^{\kappa}\|u^i\|_2\right) \\
 &\le \kappa \frac{t}{2\kappa}.
 \end{aligned}
\end{equation*}

Now, let $\mathcal{S}'_{1,2}$ be the  set of nonzero $\X = \X_1 + \X_2$ such that each $\X_i \in S_1 \cup \{{ \bf 0}\}$ and has $\langle \X_1,\X_2\rangle=0$. For a given $\mathcal{X}\in \mathcal{S}_{1,2}'$, we may set $\tilde X = \tilde X_1 + \tilde X_2$, where for $i=1,2$ $\tilde X_i$ is the best $(\mathcal{N}_1 \cup \{0\})$-approximation of $\X_i$, and
note that \begin{equation*}
\|\X - \tilde \X\|_F \le \|\X_1 - \tilde \X_1\|_F + \|\X_2 - \tilde \X_2\|_F \le t.
\end{equation*}
Thus, $\mathcal{N}_2\coloneqq (\mathcal{N}_1\cup\{{ \bf 0}\})+(\mathcal{N}_1+\cup\{{ \bf 0}\})=\{\X_1+\X_2~|~X_1, X_2\in\mathcal{N}_1+\cup\{{ \bf 0}\}\} \setminus \{ {\bf 0} \}$ is a $t$-net of $\mathcal{S}'_{1,2}$ with cardinality at most $\left(\left(\frac{6\kappa}{t}\right)^{\kappa n} + 1\right)^2$. Lastly, we note that each element of $\mathcal{S}'_{1,2}$ has norm at least one and that $\mathcal{S}_{1,2}$ is the projection of $\mathcal{S}_{1,2}'$ onto the unit sphere. Therefore, the projection of $\mathcal{N}_2$ onto the unit sphere is $t$-net for $\mathcal{S}_{1,2}$.
\end{proof}





The following technical lemma will be used in the proof of Lemma \ref{lem: CoverfancyB1}.

\begin{lemma}
\label{lem:setsetcover}
Let $A \subseteq B \subseteq \R^n$, and suppose that $C \subseteq B$ is an $\epsilon/2$-net of $B$.  Then, there exists an $\epsilon$-net $C' \subseteq A$ of $A$ with cardinality $|C'| \leq |C|$.
\end{lemma}

\begin{proof}
We will construct $C'$ from $C$ as follows.  First, let $\tilde{C}$ be the subset of $C$ whose elements are all at least $\epsilon/2$ away from $A$,
$$\tilde{C} \coloneqq \left \{ {\bf x} ~\big|~ {\bf x} \in C ~{\rm and}~ \inf_{{\bf y} \in A} \| {\bf x} - {\bf y} \|_2 \geq \epsilon/2 \right \}.$$
Next, for each ${\bf x} \in C \setminus \tilde{C}$ let ${\bf x}' \in A$ be any point of $A$ satisfying $\| {\bf x} - {\bf x}' \|_2 < \epsilon/2$, and then set
$$C' \coloneqq \bigcup_{{\bf x} \in C \setminus \tilde{C}} {\bf x}' \subseteq A.$$
Note that $|C'| \leq |C|$ by construction.

To see that $C'$ is an $\epsilon$-net of $A$, choose any ${\bf y} \in A \subseteq B$ and let ${\bf x} \in C$ be a point satisfying $\| {\bf y} - {\bf x}\| < \epsilon/2$.  Noting that ${\bf x} \notin \tilde{C}$, we can see that there is a ${\bf x}' \in C'$ such that $\| {\bf x} - {\bf x}' \|_2 < \epsilon/2$.  Therefore, by the triangle inequality,
$$\| {\bf y} - {\bf x}' \|_2 \leq \| {\bf y} - {\bf x} \|_2 + \| {\bf x} - {\bf x}' \|_2 < \epsilon.$$
This establishes the desired result.
\end{proof}

\begin{proof}[The proof of Lemma \ref{lem: CoverfancyB1}]

Our argument is based on the proof of \cite[Lemma 5]{rauhut2017low}, with necessary modifications to account for the fact that the tensor factors are not orthogonal. 

\textbf{Part 1: Construction of the net.} An arbritrary element of  $\mathcal{B}_{R,\mu,\theta,\mathbf{r}}$ can be written as $\X=\mathcal{C}\times_1V^1\times_2\ldots\times_d V^d$ where the core tensor $\mathcal{C}\in\mathbb{R}^{r\times\ldots\times r}$ is a $d$-mode tensor with Frobenius norm one and the factor matrices $V^i\in\mathbb{R}^{n\times r}$ have columns ${\bf v}^i_j$ with norm at most $R$. The set of all core tensors satisfying the orthogonality condition (d) is isometric to a subset of the unit ball in $\R^{r^d}$. Therefore, the admissible core tensors admit an $\eps_1$-net $\mathcal{N}_1$ of the cardinality at most $(3/\eps_1)^{r^d}$ by \cite[Lemma 1]{rauhut2017low}.  
We define the $\| \cdot \|_{1,2}$ norm by $\|V\|_{1,2} \coloneqq \max_j\|{\bf v}_j\|_2$ and note that by construction we have $\|V^i\|_{1,2}\leq R$. Therefore, our admissible factor matrices satisfying condition (b) have an  $\eps_2$-net (with respect to the $\|\cdot\|_{1,2}$ norm) of the cardinality at most $(3R/\epsilon_2)^{nr}$ again by \cite[Lemma 1]{rauhut2017low}. 
We now define  
\begin{equation*}
    \mathcal{N} \coloneqq\left\{\overline{\mathcal{C}}\times_1\overline{V}^1\ldots\times_d\overline{V}^d: \overline{\mathcal{C}}\in \mathcal{N}_1 \text{ and } \overline{V}^i\in \mathcal{N}_2 \right\}, \quad |\mathcal{N}| \le \left(\frac{3}{\epsilon_1}\right)^{r^d}\left(\frac{3R}{\epsilon_2}\right)^{dnr}.
\end{equation*}

Going forward we will prove that $\mathcal{N}$ above is an $\epsilon/2$-net of  $\mathcal{B}_{R,\mu,0,\mathbf{r}}$ for suitable choices of $\epsilon_1$ and $\epsilon_2$.  The result will then follow from noting $\mathcal{B}_{R,\mu,\theta,\mathbf{r}} \subseteq \mathcal{B}_{R,\mu,0,\mathbf{r}}$ and applying Lemma~\ref{lem:setsetcover}.\\

\textbf{Part 2: Term by term approximation.} 
Let us take an arbitrary element of $\mathcal{X}\in \mathcal{B}_{R,\mu,0,\mathbf{r}}$ and consider its component-wise approximation in $\overline{\mathcal{X}}\in\mathcal{N}$:
\begin{equation*}
\mathcal{X}=\mathcal{C}\times_1 V^1\times_2\ldots \times_d V^d \in  \mathcal{B}_{R,\mu,0,\mathbf{r}} \quad  \text{ and } \quad 
\overline{\mathcal{X}}=\overline{\mathcal{C}}\times_1 \overline{V}^1\times_2\ldots \times_d \overline{V}^d \in \mathcal{N},
\end{equation*}
where $\|\overline{\mathcal{C}}-\mathcal{C}\|_F\leq \epsilon_1$ and $\|\overline{V}^i-V^i\|_{1,2}\leq \epsilon_2$ for all $1\leq i \leq d.$
The triangle inequality implies that 
\begin{equation}\label{eq: main_covering_decomp}
    \|\overline{\mathcal{X}}-\mathcal{X}\|_F\leq \sum_{j=0}^d\|\mathcal{T}_j\|_F
\end{equation}
where $\mathcal{T}_0=(\overline{\mathcal{C}}-\mathcal{C})\times_1 \overline{V}^1\times_2\ldots \times_d \overline{V}^d$
and for $1\leq j\leq d$, 
\begin{equation*}
    \mathcal{T}_j=\mathcal{C}\times_1V^1\ldots \times_{j-1}V^{j-1}\times_j W^j\times_{j+1}\overline{V}^{j+1}\ldots\times_d  \overline{V}^d \quad \text{ for }1\leq j \leq d,
\end{equation*} where 
 $W^j \coloneqq V^j-\overline{V}^j$.\\

\textbf{Part 3: Bounding $\mathcal{T}_j$ for $\mathbf{j = 1, \ldots, d}$.} For $1\leq j \leq d$, we can expand
\begin{equation}\begin{aligned}\label{eqn: tj}
    \|\mathcal{T}_j\|_F^2 &=\sum_{t_d=1}^n\ldots\sum_{t_1=1}^n|\mathcal{T}_j(t_1,\ldots,t_d)|^2 \\
    &=\sum_{\substack{t_i=1\\ 1\leq i \leq d}}^n\left|\sum_{k_d=1}^r\ldots\sum_{k_1=1}^r \mathcal{C}(k_1,\ldots,k_d)\left[\prod_{i=1}^{j-1} v^i_{k_i}(t_i)\right]w^j_{k_j}(t_j)\left[\prod_{i=j+1}^d\overline{v}^i_{k_i}(t_i)\right]\right|^2,
\end{aligned}\end{equation}    
where ${\bf v}^i_1,\ldots,{\bf v}^i_r$ denote the columns of $V^i$, and similarly  ${\bf w}_1^i,\ldots,{\bf w}_r^i$ and $\overline{\bf v}_1^i,\ldots,\overline{\bf v}_r^i$ denote the columns of $W^i$ and $\overline{V}^i$. Exchanging the sums, we can rewrite \eqref{eqn: tj} in the following way:
\begin{equation*}\begin{aligned}
    \sum_{\substack{l_i=1\\ 1\leq i \leq d}}^r\sum_{\substack{k_i=1\\ 1\leq i \leq d}}^r
    &\mathcal{C}(k_1,\ldots,k_d)\mathcal{C}(l_1,\ldots,l_d) \sum_{\substack{t_i=1\\ 1\leq i \leq d}}^n  \left[\prod_{i=1}^{j-1}v^i_{l_i}(t_i) v^i_{k_i}(t_i)\right]w^j_{l_j}(t_j)w^j_{k_j}(t_j)\left[\prod_{i=j+1}^d\overline{v}^i_{k_i}(t_i)\overline{v}^i_{k_i}(t_i)\right] \\
    &=\sum_{\substack{l_i=1, k_i=1\\ 1\leq i \leq d}}^r
    \mathcal{C}(k_1,\ldots,k_d)\mathcal{C}(l_1,\ldots,l_d)\left[\prod_{i=1}^{j-1} \langle {\bf v}^i_{k_i},{\bf v}^i_{\ell_i}\rangle\right] \langle {\bf w}^j_{k_j},{\bf w}^j_{\ell_j}\rangle\left[\prod_{i=j+1}^d \langle \overline{\bf v}^i_{k_i},\overline{\bf v}^i_{\ell_i} \rangle\right].
\end{aligned}\end{equation*} 


We estimate scalar products by $|\langle {\bf w}^i_{k_i},{\bf w}^i_{\ell_i}\rangle| \le \|{\bf w}^i_{k_i}\|\|{\bf w}^i_{\ell_i}\| \le \eps_2^2$ and
 \begin{equation}\label{eq:v_bound}
 |\langle \overline{\bf v}^i_{k_i}, \overline{\bf v}^i_{\ell_i} \rangle|,|\langle {\bf v}^i_{k_i}, {\bf v}^i_{\ell_i} \rangle| \leq
 \begin{cases}
 R^2,  &\text{ if } k_i = l_i, \\
 \mu,  &\text{ otherwise }
 \end{cases}.
 \end{equation}
 Therefore, for any $1 \le j \le d$, we have 
 \begin{equation*}
     \begin{aligned}
 \sum_{\substack{l_i=1, k_i=1\\ 1\leq i \leq d}}^r
    &\mathcal{C}(k_1,\ldots,k_d)\mathcal{C}(l_1,\ldots,l_d)\left[\prod_{i=1}^{j-1} \langle {\bf v}^i_{k_i},{\bf v}^i_{\ell_i}\rangle\right] \langle {\bf w}^j_{k_j},{\bf w}^j_{\ell_j}\rangle\left[\prod_{i=j+1}^d \langle \overline{\bf v}^i_{k_i},\overline{\bf v}^i_{\ell_i} \rangle\right]\\
    &=\sum_{\substack{l_i=1, k_i=1\\ 1\leq i \leq d\\
    \\k_i=\ell_i \forall i\neq j}}^r\mathcal{C}(k_1,\ldots,k_d)\mathcal{C}(l_1,\ldots,l_d)\left[\prod_{i=1}^{j-1} \langle {\bf v}^i_{k_i},{\bf v}^i_{\ell_i}\rangle\right] \langle {\bf w}^j_{k_j},{\bf w}^j_{\ell_j}\rangle\left[\prod_{i=j+1}^d \langle \overline{\bf v}^i_{k_i},\overline{\bf v}^i_{\ell_i} \rangle\right]\\
    &\quad \quad \quad+\sum_{\substack{l_i=1, k_i=1\\ 1\leq i \leq d\\
    \\\exists i \neq j~{\rm s.t.}~k_i \neq \ell_i}}\mathcal{C}(k_1,\ldots,k_d)\mathcal{C}(l_1,\ldots,l_d)\left[\prod_{i=1}^{j-1} \langle {\bf v}^i_{k_i},{\bf v}^i_{\ell_i}\rangle\right] \langle {\bf w}^j_{k_j},{\bf w}^j_{\ell_j}\rangle\left[\prod_{i=j+1}^d \langle \overline{\bf v}^i_{k_i},\overline{\bf v}^i_{\ell_i} \rangle\right] \\
    &\leq R^{2(d-1)}\eps_2^2\sum_{\substack{l_i=1, k_i=1\\ 1\leq i \leq d\\
    \\k_i=\ell_i \forall i\neq j}}^r\mathcal{C}(k_1,\ldots,k_d)\mathcal{C}(l_1,\ldots,l_d)
    +\mu R^{2(d-2)}\eps_2^2\sum_{\substack{l_i=1, k_i=1\\ 1\leq i \leq d}}^r \left|\mathcal{C}(k_1,\ldots,k_d)\mathcal{C}(l_1,\ldots,l_d) \right|
 \end{aligned}
 \end{equation*}
since $\mu < 1 \leq R$ by assumption.  Hence, we have
\begin{equation*}\begin{aligned}
    \|\mathcal{T}_j\|_F^2
    ~&\leq~ R^{2(d-1)}\eps_2^2\sum_{\substack{k_i=1 \\ 1\le i \le d, i\neq j}}^r\sum_{k_j=1}^r\sum_{\ell_j=1}^r\mathcal{C}(k_1,\ldots,k_j, \ldots, k_d)\mathcal{C}(k_1,\ldots,l_j, \ldots, k_d) \\
     & \quad \quad + \mu R^{2(d-2)}\eps_2^2 \sum_{\substack{l_i=1, k_i=1\\ 1\leq i \leq d}}^r \left|\mathcal{C}(\ell_1,\ldots,\ell_d)\mathcal{C}(k_1,\ldots,k_d) \right| \\
     &= R^{2(d-1)}\eps_2^2\sum_{k_i=1}^d|\mathcal{C}(k_1,\ldots,k_j,\ldots,k_d)|^2 + \mu R^{2(d-2)}\eps_2^2 \left(\sum_{\substack{\ell_i=1 \\1 \le i \le d}}^r \left|\mathcal{C}(\ell_1,\ldots,\ell_d) \right|\right)^2,
\end{aligned}\end{equation*}
where in the last step we have used the fact that by the orthogonality property (d)
$$
\sum_{\substack{k_i=1 \\ i\neq j}}^r\mathcal{C}(k_1,\ldots,k_j, \ldots, k_d)\mathcal{C}(k_1,\ldots,l_j, \ldots, k_d) = 0 \quad \text{ unless } k_j = \ell_j.
$$
to see that
$$
\sum_{\substack{k_i=1 \\ i\neq j}}^r\sum_{k_j=1}^r\sum_{\ell_j=1}^r\mathcal{C}(k_1,\ldots,k_j, \ldots, k_d)\mathcal{C}(k_1,\ldots,l_j, \ldots, k_d) = \sum_{k_i=1}^d|\mathcal{C}(k_1,\ldots,k_j,\ldots,k_d)|^2.
$$

Recalling that all of our core tensors are unit norm, and appealing to Cauchy-Schwarz now allows us to see that
\begin{equation}\begin{aligned}
    \|\mathcal{T}_j\|_F^2
    ~&\leq~ R^{2(d-1)}\eps_2^2 + \mu R^{2(d-2)}\eps_2^2 r^{d} \| \mathcal{C} \|_{ F}^2\\
    \label{eqn: Tj}
    &\le \epsilon_2^2R^{2d-4}(R^2+\mu r^d).
\end{aligned}\end{equation}


\bigskip
\textbf{Part 4: Bounding $\mathcal{T}_0$.} We note that for any $1\leq i \leq d$, the $\| \cdot \|_F \to \| \cdot \|_F$  operator norm of the   operator $\mathcal{X}\rightarrow \mathcal{X}\times_i V^i$ is the same as the $\ell_2$-operator norm of the matrix $V^i$ acting on $\mathbb{R}^r$. Next, we observe that for all ${\bf x} = (x_1, \ldots, x_r) \in \mathbb{R}^r$ with $\|{\bf x}\|_2=1,$ we have 
\begin{equation*}
    \|V^i {\bf x}\|_2^2= \sum_{k,l = 1}^r \langle {\bf v}^i_k,{\bf v}^i_l\rangle x_l x_k=\sum_{k = 1}^r \langle {\bf v}^i_k,{\bf v}^i_k\rangle x_k^2 + \sum_{k\neq l = 1}^r \langle {\bf v}^i_k,{\bf v}^i_l\rangle x_k x_l.
\end{equation*}
Thus, bounding the coefficients by \eqref{eq:v_bound} and using Cauchy–Schwarz we have that 
\begin{equation*}
    \|V^i {\bf x}\|_2^2\leq R^2 \sum_{k = 1}^r x_k^2 + \mu \sum_{k\neq l = 1}^r |x_k| |x_l| \leq R^2 \| {\bf x}\|^2_2 + \mu \left( \sum_{k = 1}^r |x_k| \right)^2 \leq  (R^2+r\mu)\| {\bf x}\|^2_2.
\end{equation*}
Therefore, since $\|\mathcal{C}-\overline{\mathcal{C}}\|_F^2\leq \epsilon_1$,
\begin{equation*}
    \|\mathcal{T}_0\|^2_F = \|(\mathcal{C}-\overline{\mathcal{C}})\times_1\overline{V}_1\times_2\ldots\times_d \overline{V}_d\|^2_F \leq (R^2+r \mu)^d\epsilon_1^2.
\end{equation*}

\textbf{Part 5: Conclusion and cardinality estimate.} 
By \eqref{eq: main_covering_decomp}, we get that $ \|\mathcal{X}-\overline{\mathcal{X}}\|_F  \le \eps/2$ if each $\|\mathcal{T}_j\| \le \eps/2(d+1)$. That is, taking $(R^2+\mu r)^{d/2}\epsilon_1 \coloneqq \eps/2(d+1)$ and $\epsilon_2R^{d-2}(R^2+\mu r^d)^{1/2} \coloneqq \eps/2(d+1)$, we get
\begin{equation}\begin{aligned}
|\mathcal{N}| \le  \left(\frac{3}{\epsilon_1}\right)^{r^d}\left(\frac{3R}{\epsilon_2}\right)^{dnr} &= \left(\frac{6(d+1)}{\eps}\right)^{r^d + rnd} \left((R^2+\mu r)^{d/2}\right)^{r^d}\left(R^{d-1}(R^2+\mu r^d)^{1/2}\right)^{dnr}.
\end{aligned}\end{equation}
This concludes the proof of Lemma~\ref{lem: CoverfancyB1}.
\end{proof}


 
 \section{The proof of Proposition \ref{thm: sup_chaos}}\label{sec: The proof of the Proposition}

Proposition~\ref{thm: sup_chaos} follows by essentially repeating the proof of Theorem 2 of \cite{rauhut2017low}. We include the argument here for completeness. The key probabilistic component of the proof is the supremum of chaos inequality proved in \cite{Krahmer2014Suprema}:
 \begin{theorem}[\cite{Krahmer2014Suprema}, Theorem 3.1]\label{gaus_chaos}
 Let $\mathcal{A}$ be a collection of matrices, which size is measured through the following three quantities $E, V$ and $U$ defined as
 \begin{equation}\label{evu} \begin{aligned}
 E(\mathcal{A})&\coloneqq\gamma_2(\mathcal{A})(\gamma_2(\mathcal{A})+d_F(\mathcal{A}))+d_F(\mathcal{A})d_{2}(\mathcal{A})\\
 V(\mathcal{A})&\coloneqq d_{2}(\mathcal{A})(\gamma_2(\mathcal{A})+d_F(\mathcal{A})),\quad\text{and}\quad \\
 U(\mathcal{A})&\coloneqq d_{2}^2(\mathcal{A}),
 \end{aligned}
  \end{equation}
   where
\begin{equation*}
d_F(\mathcal{A})\coloneqq\sup_{A\in\mathcal{A}}\|A\|_F,\quad d_{2}(\mathcal{A})\coloneqq\sup_{A\in\mathcal{A}}\|A\|_{2\rightarrow2}\end{equation*}
and $\gamma_2(\mathcal{A}) = \gamma_2(\mathcal{A},\|\cdot\|_{2\rightarrow2})$ is Talagrand's functional. Let $\vct{\xi}$ be a sub-gaussian random vector whose entries $\xi_j$ are independent, mean-zero and variance-1. 
 Then, for $t>0,$
 \begin{equation*}
 \mathbb{P}\left(\sup_{A\in\mathcal{A}} \left| \|A\vct{\xi}\|_2^2-\mathbb{E}\|A\vct{\xi}\|_2^2\right|\geq c_1E(\mathcal{A})+t\right)\leq 2\exp\left(-c_2\min\left\{\frac{t^2}{V(\mathcal{A})^2},\frac{t}{U(\mathcal{A})}\right\}\right),
 \end{equation*}
 where the constants $c_1,c_2$ depend only on the sub-gaussian constant $L$.
 \end{theorem}


 \begin{proof}[The proof of Proposition \ref{thm: sup_chaos}] Observe that
$$\eps_S \coloneqq \sup_{\mathcal{X} \in S, \|\mathcal{X}\|= 1}\left|\|A(\vect(\mathcal{X}))\|^2 - 1\right| = \sup_{\mathcal{X} \in S, \|\mathcal{X}\|= 1}\left|\|V_{\X}\vct{\xi}\|^2 - \E\|V_{\X}\vct{\xi}\|^2\right|,$$
where $\xi \in \R^{n^\kappa m}$ is a vector with i.i.d. sub-gaussian entries, and $V_{\X} \in \R^{m \times n^\kappa m}$ is a block-diagonal matrix 
$$
V_{\X} = \frac{1}{\sqrt{m}} 
\begin{bmatrix}
\vect(\X)^T & 0 & \ldots & 0\\
0 & \vect(\X)^T & \ldots & 0\\
\ldots & \ldots & \ldots & \ldots\\
0 & 0 & \ldots & \vect(\X)^T
\end{bmatrix}.
$$
Let $\V = \V(\S)$ be the set of all such matrices $V_{\X}$ where $\X \in \S$. We will now apply Theorem~\ref{gaus_chaos}. It is easy to check (see \cite[Theorem 3]{rauhut2017low}) that 
\begin{equation}\label{d_2_F}
d_2(\V) = m^{-1/2} \text{ and } d_F(\V) = 1,
\end{equation}
therefore, by Theorem~\ref{gaus_chaos},
\begin{equation}\label{sc-estimate}
\P\left\{\eps_{\S} \ge c_1\left(\gamma_2^2 + \gamma_2 + \frac{1}{\sqrt{m}}\right) + t\right\} \le 2 \exp\left(-c_2 \min\left\{mt, \frac{\sqrt{m}t^2}{1 + \gamma_2}\right\}\right) \text{ for all } t > 0,
\end{equation}
where $\gamma_2 = \gamma_2(\V).$

For any set $\S_0$, the Talagrand functional $\gamma_2(\S_0)$  is a functional of  $\S_0$ which can be bounded by the Dudley-type integral
\begin{equation}\label{dudley-est}
\gamma_2 (\S_0)\le C \int_0^{d_2(\S_0)}\sqrt{\ln{\mathcal{N}(\S_0, \| \cdot \|_{2 \rightarrow 2}, t)}}~dt.  \end{equation}
We need to consider $\S_0 = \V(\S)$. We note that due to \eqref{d_2_F}, 
\begin{equation*}
\mathcal{N}(\mathcal{V}({\mathcal{S}}),\|\cdot\|_{2 \rightarrow 2},u)\leq \mathcal{N}(\mathcal{\S},\|\cdot\|_{\rm F},\sqrt{m}u),
\end{equation*}
and so by a change of variables
$$
    \gamma_2(\V(\S))\leq C\frac{1}{\sqrt{m}}\int_0^{1}\sqrt{\ln(\mathcal{N}(\mathcal{S},\|\cdot\|_{\rm F},u))}~du.
$$
This implies that $\gamma_2(\V) \le C\tilde{C}^{-1/2}\eps$ by the main condition on $m$ given in the statement of Proposition~\ref{thm: sup_chaos}. Choosing $\tilde{C}$ large enough, this ensures that $\gamma_2 \le \eps/6c_1$ for any $c_1 > 1$.

Taking $t = \eps/2$ and using $\eps < 1$,  we can now rewrite \eqref{sc-estimate} as
$$\P\left\{\eps_{\S} \ge c_1\left(\left[\frac{\eps}{6c_1}\right]^2 + \frac{\eps}{6c_1} + \frac{1}{\sqrt{m}}\right) + \frac{\eps}{2} \right\} \le 2 \exp\left(-C_2m\eps^2\right)
$$
If $\tilde{C}$ is chosen large enough, then $m \ge \tilde{C} \eps^{-2}\ln{\eta^{-1}}$ implies that the probability bound is at most $\eta$ and $m \ge \tilde{C}\eps^{-2}$ implies that $m^{-1/2} \le \eps/6c_1$. This concludes the proof of Proposition~\ref{thm: sup_chaos}.
\end{proof}

\end{document}